\documentclass{svproc}

\usepackage{amsmath,amssymb,amsfonts}
\usepackage{graphicx}
\usepackage{url}
\usepackage{bm}
\usepackage{enumitem}
\usepackage{mathrsfs}
\usepackage{cite}
\usepackage[hidelinks]{hyperref}

\numberwithin{theorem}{section}
\numberwithin{lemma}{section}
\numberwithin{proposition}{section}
\numberwithin{corollary}{section}
\numberwithin{definition}{section}  
\numberwithin{remark}{section}

\spnewtheorem{assumption}[theorem]{Assumption}{\bfseries}{\itshape}
\numberwithin{equation}{section}

\newcommand{\NDelta}{N_\Delta}
\newcommand{\NC}{N_{\mu_C}}

\newcommand{\RV}{\mathrm{RV}}

\begin{document}
\mainmatter

\title{Rigidity of Spectral Encodings under \\ Weyl Growth Conditions}
\author{Anton Alexa}
\tocauthor{Anton Alexa} 

\institute{Independent Researcher, Chernivtsi, Ukraine\\
\email{mail@antonalexa.com}}

\maketitle 

\begin{abstract}
We prove that the geometric Weyl bulk-density exponent $(d-2)/2$ rigidifies
spectral encodings $C=\pi-\phi(\lambda)$ in the O-regularly varying class:
the bulk power law forces $\phi\in\mathrm{RV}_1$ (asymptotic linearity).
For polynomial-type encodings
$C=\pi-\epsilon\lambda^k L(\lambda)$ with $L\in\mathrm{RV}_0$,
this yields the unique admissible exponent $k=1$.
The affine encoding then gives
$N_{\mu_C}(C)\sim\gamma_d\,\epsilon^{-d/2}(\pi-C)^{d/2}$ as $C\to-\infty$,
allowing recovery of $d$ and $\gamma_d$ from bulk encoded data.
This transfer is stable under perturbations $\delta(\lambda)=o(\lambda)$,
with explicit slowly varying error control.
We further formalize asymptotic spectral equivalence classes:
if $\phi\in\mathrm{RV}_k$, the induced map scales asymptotic spectral dimension
as $d_{\mathrm{as}}\mapsto d_{\mathrm{as}}/k$; hence dimension preservation
is equivalent to $\phi\in\mathrm{RV}_1$, with strict affine normalization at
first order when $L(\lambda)\to1$.
\end{abstract}

\keywords{Weyl law, spectral encoding rigidity, affine rigidity, O-regular variation, Tauberian theorems, bulk--edge asymptotics, asymptotic spectral dimension, spectral equivalence classes, Kre{\u\i}n strings, stability}

\section{Introduction}\label{sec:intro}

We study the following rigidity question:
which spectral encodings of the Laplace spectrum preserve the geometric
Weyl density exponent?

Let $(M^d,g)$ be a smooth compact Riemannian manifold without boundary, with
Laplace eigenvalues $0=\lambda_0\le\lambda_1\le\cdots\to\infty$.
A spectral encoding is a reparametrization $C=f(\lambda)$ that maps
the spectrum to a new variable; the induced counting function $N_{\mu_C}$ and
bulk density $\rho_{\mathrm{bulk}}$ carry whatever spectral information survives
the reparametrization.
For this framework to be useful in geometric inverse problems, the encoding must
be rigid: it must neither destroy nor artificially create the exponent
$(d-2)/2$ that signatures dimension in Weyl's law
\cite{Hoermander1968,SafarovVassiliev1997,Ivrii1998}.

Our main result establishes rigidity within the full class of O-regularly varying
encodings \cite{BGT87} (Theorem~\ref{thm:ORV-uniqueness}):
the bulk power-law constraint forces $\phi\in\mathrm{RV}_1$, i.e.,
asymptotic linearity of the encoding.
In the polynomial case $C=\pi-\epsilon\lambda^k L(\lambda)$ with
$L\in\mathrm{RV}_0$, this specializes to $k=1$:
any $k\neq1$ distorts the exponent to $(d-2)/(2k)$, breaking spectral-dimension
recovery.

As a structural consequence of the uniquely correct affine encoding
$C=\pi-\epsilon\lambda$, the edge-variable Weyl law
\begin{equation*}
N_{\mu_C}(C)\sim\gamma_d\,\epsilon^{-d/2}(\pi-C)^{d/2},\qquad C\to-\infty,
\end{equation*}
holds, allowing dimension $d$ and Weyl constant $\gamma_d$ to be recovered from
one-dimensional bulk data.
This transfer is stable under perturbations $\delta(\lambda)=o(\lambda)$
(Proposition~\ref{prop:stability}), with error rates governed by the slowly
varying factor $\eta=\delta/\lambda\in\mathrm{RV}_0$.

To make the rigidity statement structural (rather than only model-specific),
we introduce asymptotic spectral equivalence classes determined by regular
variation exponent and define the encoding-induced map on these classes.
For $\phi\in\RV_k$, this map scales asymptotic spectral dimension by
$d_{\mathrm{as}}\mapsto d_{\mathrm{as}}/k$, so the dimension-preserving
case is exactly $\phi\in\RV_1$ (with strict affine normalization when
$L(\lambda)\to1$).

Beyond uniqueness and stability, we prove that the full discrete spectral measure
with multiplicities --- capturing clustering on symmetric spaces --- is completely
and uniquely realizable by a Kre\u{\i}n string
(Theorem~\ref{thm:MSL-realization}, \cite{EckhardtTeschl2013}).
No classical smooth Sturm--Liouville operator can do this for $d>1$
(Lemma~\ref{lem:no-classical-SL}, \cite{Titchmarsh1962,LevitanSargsjan1991}),
so the Kre\u{\i}n model is the natural one-dimensional realization of Riemannian
spectra.

\begin{theorem}[Affine Rigidity Theorem]\label{thm:main}
Let $(M^d,g)$ be a smooth compact Riemannian manifold without boundary.
\begin{enumerate}[label=\textup{(\roman*)},itemsep=3pt]
\item \textup{(O-regular variation rigidity.)}
Among all O-regularly varying encodings $C=\pi-\phi(\lambda)$,
the requirement that the bulk density exponent equals $(d-2)/2$
forces $\phi\in\mathrm{RV}_1$, i.e.\ asymptotic linearity of the encoding.

\item \textup{(Polynomial special case.)}
In particular, among all polynomial-type encodings $C=\pi-\epsilon\lambda^k L(\lambda)$
with $k>0$ and $L\in\mathrm{RV}_0$, only the asymptotically affine case
$k=1$ preserves the geometric bulk density exponent $(d-2)/2$.
Any $k\neq1$ distorts the exponent to $(d-2)/(2k)\neq(d-2)/2$,
making spectral-dimension recovery impossible.
\end{enumerate}
\end{theorem}

\begin{corollary}[Edge-variable Weyl law and stability]\label{cor:edge-weyl}
Under the affine encoding $C=\pi-\epsilon\lambda$:
\begin{enumerate}[label=\textup{(\roman*)},itemsep=3pt]
\item \textup{(Weyl transfer.)}
In the bulk regime $C\to-\infty$,
\begin{equation}
N_{\mu_C}(C)\sim \gamma_d\,\epsilon^{-d/2}(\pi-C)^{d/2},\qquad
\rho_{\mathrm{bulk}}(C)\sim \tfrac{d}{2}\,\gamma_d\,\epsilon^{-d/2}(\pi-C)^{(d-2)/2},
\end{equation}
enabling exact recovery of $d$ and $\gamma_d$ from one-dimensional
encoded data in the bulk regime ($C\to-\infty$).

\item \textup{(Stability.)}
The bulk exponents are preserved under perturbations
$C=\pi-\epsilon\lambda+\delta(\lambda)$ with $\delta(\lambda)=o(\lambda)$,
with error rates governed by the slowly varying factor $\eta=\delta/\lambda\in\mathrm{RV}_0$.
\end{enumerate}
\end{corollary}

\paragraph{Context.}
The edge-variable framework was introduced informally in
\cite{alexa2025-spectral_classification} (sec.~VI.A), where spectral density in
the edge variable was proposed as a dimensional signature.
The present paper establishes the rigidity of this choice: the affine encoding
is not merely convenient but is the unique polynomial-type
reparametrization compatible with spectral-dimension recovery.
The proofs rely on Tauberian tools from regular variation theory \cite{BGT87}
and on the Kre\u{\i}n string correspondence \cite{EckhardtTeschl2013}.

\begin{remark}[Terminology]
We reserve edge for the geometric endpoint $C\uparrow\pi$, while
bulk refers to the high-energy regime $C\to-\infty$.
Note that $C=\pi$ is the image of $\lambda_0=0$ under the encoding
$C=\pi-\epsilon\lambda$: it is the spectral edge of the Riemannian Laplacian
in the $C$-variable (the encoding reverses orientation, so the lower
spectral boundary $\lambda=0$ maps to the upper boundary $C=\pi$).
The terminology is inherited from the companion framework
\cite{alexa2025-spectral_classification}, where $C=\pi$ is the natural upper
limit of the deformation operator $\hat{C}$ and the variable $C$ is called the
edge variable; the Weyl-type asymptotics occur in the bulk
($C\to-\infty$) regime of this variable.
For modern overviews of Weyl asymptotics and spectral Tauberian methods,
see \cite{Grieser2009}.
\end{remark}
\section{Preliminaries and standing assumptions}\label{sec:prelim}

This section fixes notation and records the standing assumptions used throughout.
We introduce the Laplace spectrum on a compact Riemannian manifold, the
affine encoding $C=\pi-\epsilon\lambda$ and the induced edge-variable measures,
the composition identity that connects edge-variable counting to the classical
counting function, and the regular variation toolkit from which the uniqueness
and stability arguments draw.

\subsection{Geometry and Laplace spectrum}\label{subsec:geom-laplace}

We work throughout on a smooth compact Riemannian manifold without boundary;
the Laplace--Beltrami operator has discrete spectrum with finite multiplicities.

\begin{assumption}[Smoothness and spectrum]\label{ass:smoothness}
Throughout, $M^d$ is a smooth, connected, compact Riemannian manifold without boundary, and $-\Delta_g$ denotes the nonnegative Laplace--Beltrami operator on $L^2(M)$. It is essentially self-adjoint on $C^\infty(M)$ and has purely discrete spectrum $0=\lambda_0\le\lambda_1\le\cdots\to\infty$, each eigenvalue with finite multiplicity $m_n$; \cite[Ch.~III]{Chavel1984}, \cite[§§2--3]{Rosenberg1997}, \cite[Ch.~XIII]{ReedSimonIV}.
\end{assumption}

\subsection{Spectral measures and counting}\label{subsec:measures}

We encode the spectrum as a discrete measure and record the c\`adl\`ag properties
of the associated counting function used throughout.

\begin{definition}[Spectral measure and counting function]
Let $(M^d,g)$ be a compact Riemannian manifold without boundary, and let 
\begin{equation}
-\Delta_g \phi_n = \lambda_n \phi_n, \qquad \|\phi_n\|_{L^2(M)}=1,
\end{equation}
with eigenvalues $0=\lambda_0 \le \lambda_1 \le \cdots \to \infty$ and multiplicities $m_n$.  
The associated \emph{spectral measure} is
\begin{equation}
\mu_\Delta := \sum_n m_n \, \delta_{\lambda_n},
\end{equation}
and its \emph{counting function} is
\begin{equation}
N_\Delta(\Lambda) := \mu_\Delta((-\infty,\Lambda]) = \sum_{\lambda_n \le \Lambda} m_n,
\end{equation}
where eigenvalues are repeated according to their multiplicities $m_n$
\cite[Thm.~VII.2]{ReedSimonI}.
\end{definition}

\begin{lemma}[Counting function properties]\label{lem:counting-properties}
$N_\Delta$ is nondecreasing, \\ right--continuous with left limits (c\`adl\`ag), and has
jumps of size $m_n$ at the eigenvalues $\lambda_n$. Hence $N_\Delta\in BV_{\mathrm{loc}}([0,\infty))$
and is the Stieltjes distribution function of $\mu_\Delta$.
\end{lemma}
\begin{proof}
Since $N_\Delta(\Lambda)=\mu_\Delta((-\infty,\Lambda])$ with a purely atomic measure,
all stated properties are immediate.
\qed
\end{proof}

\subsection{Affine edge encoding and edge variable}\label{subsec:encoding}

Fix $\epsilon>0$ and define $C_n=\pi-\epsilon\lambda_n$, the pushforward measure
$\mu_C:=\sum_n m_n\,\delta_{C_n}$, the edge variable $y:=\pi-C\downarrow0$, and
\begin{equation}
\NC(C):=\mu_C([C,\pi))=\#\{n:\ C_n\ge C\}.
\end{equation}
We write
\begin{equation}
\rho_{\mathrm{edge}}(C):=-\,\frac{d}{dC}\NC(C)\ (\ge0)
\end{equation}
in the Stieltjes/distributional sense (or after a standard smoothing). When emphasizing the bulk regime $C\to-\infty$, we use the synonymous notation
\begin{equation}
\rho_{\mathrm{bulk}}(C):=-\,\frac{d}{dC}\NC(C),
\end{equation}
i.e.\ the same distributional derivative, with the name indicating the asymptotic regime.

\begin{lemma}[Monotonicity of the encoding]\label{lem:mono-encoding}
Fix $\epsilon>0$ and set $C=\pi-\epsilon\lambda$. Then the map $\lambda\mapsto C$ is strictly decreasing. Consequently,
\begin{equation}
\lambda_i<\lambda_j \ \Longleftrightarrow\ C_i>C_j,\qquad
\lambda_i=\lambda_j \ \Longleftrightarrow\ C_i=C_j,
\end{equation}
so the pushforward preserves multiplicities and order.
\end{lemma}
\begin{proof}
Immediate from $dC/d\lambda=-\epsilon<0$.
\qed
\end{proof}

\begin{remark}
This strict monotonicity is the key reason why multiplicities and ordering transfer exactly to the edge-variable encoding.
\end{remark}

\subsection{Exact composition identity}\label{subsec:exact-identity}

The following lemma is the algebraic core of the paper: it shows that the
edge-variable counting function is identical to the classical counting function
evaluated at a rescaled argument, making every asymptotic result about $N_\Delta$
immediately available for $N_{\mu_C}$.

\begin{lemma}[Exact composition identity]\label{lem:exact}
For every $C<\pi$ one has
\begin{equation}
\NC(C)=\#\{n:\lambda_n\le(\pi-C)/\epsilon\}
= N_\Delta\!\big((\pi-C)/\epsilon\big).
\end{equation}
\end{lemma}
\begin{proof}
Since $C=\pi-\epsilon\lambda$ is strictly decreasing,
$\{n: C_n\ge C\}=\{n:\lambda_n\le(\pi-C)/\epsilon\}$, with multiplicities.
\qed
\end{proof}

\subsection{Regular variation toolkit}\label{subsec:RV}

A positive function $f$ is in $\RV_\alpha$ if $f(tu)/f(t)\to u^\alpha$ as $t\to\infty$ for each $u>0$.
If $f(x)=x^\alpha \ell(x)$ with $\ell\in\RV_0$ (slowly varying), then $f\in\RV_\alpha$; \cite[§§1.2--1.7]{BGT87}.

\subsection{Standing assumptions}\label{subsec:assumptions}

For the reader's convenience we collect all hypotheses used in the paper
into a single assumption block.

\begin{assumption}[Standing framework]\label{ass:framework}
We assume throughout:

(i) $(M^d,g)$ satisfies Assumption~\ref{ass:smoothness};

(ii) Weyl's law holds
\begin{equation}
N_\Delta(\Lambda)\sim \gamma_d\,\Lambda^{d/2}\qquad(\Lambda\to\infty),
\end{equation}
with the classical remainder $O(\Lambda^{(d-1)/2})$ when available; \cite{Hoermander1968,DuistermaatGuillemin1975,Seeley1978,SafarovVassiliev1997,Ivrii1998} and modern expositions \cite{Grieser2009,Sogge2014,Zelditch2017};

(iii) a fixed scaling parameter $\epsilon>0$ defines the affine encoding $C=\pi-\epsilon\lambda$;

(iv) we write $\omega_d:=\mathrm{Vol}(B_d)=\pi^{d/2}/\Gamma(\tfrac d2+1)$ for the volume of the unit ball in $\mathbb{R}^d$;

(v) in Tauberian steps we tacitly use standard smoothings and the monotone density theorem for regularly varying functions \cite[Thm.~1.7.2]{BGT87}.

For manifolds with boundary we appeal to the two-term Weyl expansion; \cite[Ch.~1]{Ivrii1998}, \cite[Ch.~8]{SafarovVassiliev1997}.
\end{assumption}

\section{Weyl Law in the Edge Variable}\label{sec:edge-weyl}

The rigidity theorem (Theorem~\ref{thm:main}) singles out the affine encoding as
the unique preserver of the geometric density exponent.
This section records the positive side of that rigidity: what the affine encoding
actually delivers.
The key identity $N_{\mu_C}(C)=N_\Delta((\pi-C)/\epsilon)$
(Lemma~\ref{lem:exact}) reduces every asymptotic question about $N_{\mu_C}$
to the corresponding question about $N_\Delta$, so the classical Weyl law transfers
immediately; the remainder estimate, uniformity in $\epsilon$, and the two-term
boundary case are then recorded as separate propositions.

\subsection{Edge-variable Weyl (Abelian direction)}\label{subsec:abelian}

The following theorem establishes the Weyl law in the edge variable and its
converse; the converse is the reconstruction step that lets one read off $d$
and $\gamma_d$ from bulk edge-variable data.

\begin{theorem}[Edge-variable Weyl and reconstruction]\label{thm:edge-recon}
Let $(M^d,g)$ be compact without boundary and suppose the Laplace--Beltrami
counting function satisfies Weyl's law
\begin{equation}
N_\Delta(\Lambda)\sim \gamma_d\,\Lambda^{d/2}\qquad(\Lambda\to\infty),
\end{equation}
with $\gamma_d>0$. Fix $\epsilon>0$, encode $C=\pi-\epsilon\lambda$, and set
\begin{equation}
\NC(C):=\#\{n:\ C_n\ge C\}\quad(C<\pi),\qquad
\rho_{\mathrm{bulk}}(C):=-\frac{d}{dC}\NC(C)\ (\ge 0)
\end{equation}
in the Stieltjes/distributional sense (or after a standard smoothing). Then, as $C\to-\infty$ (equivalently, $(\pi-C)/\epsilon\to\infty$),
\begin{equation}\label{eq:edge-weyl-asymp}
\NC(C)\sim \gamma_d\,\epsilon^{-d/2}\,(\pi-C)^{d/2},
\end{equation}
and
\begin{equation}\label{eq:edge-density-asymp}
\rho_{\mathrm{bulk}}(C)\sim \frac{d}{2}\,\gamma_d\,\epsilon^{-d/2}\,
(\pi-C)^{\frac{d-2}{2}}.
\end{equation}
Conversely, if for some $A>0$ and $\alpha>0$,
\begin{equation}
\NC(C)\sim A\,(\pi-C)^\alpha\qquad(C\to-\infty),
\end{equation}
and $\NC$ is ultimately monotone, then necessarily
\begin{equation}
d=2\alpha,\qquad \gamma_d=A\,\epsilon^{d/2}.
\end{equation}
\end{theorem}

\begin{proof}
Since $C=\pi-\epsilon\lambda$ is strictly decreasing, the exact identity
\begin{equation}
\NC(C)=N_\Delta\!\Big(\frac{\pi-C}{\epsilon}\Big)\qquad(C<\pi)
\end{equation}
holds with multiplicities (Lemma~\ref{lem:exact}). With $\Lambda=(\pi-C)/\epsilon$ and Weyl's law we get \eqref{eq:edge-weyl-asymp}:
\begin{equation}
\NC(C)=N_\Delta(\Lambda)\sim \gamma_d\,\Lambda^{d/2}
=\gamma_d\,\epsilon^{-d/2}(\pi-C)^{d/2}.
\end{equation}
For the density, put $y:=\pi-C\to\infty$ and $F(y):=\NC(\pi-y)=N_\Delta(y/\epsilon)$.
Then
\begin{equation}
F(y)\sim \gamma_d\,\epsilon^{-d/2}\,y^{d/2}=y^{d/2}\,\ell(y),
\quad\text{with }\ \ell(y)\equiv \gamma_d\,\epsilon^{-d/2}\in \RV_0.
\end{equation}
$F$ is monotone as a counting function; we apply the smoothing of
Prop.~\ref{prop:smooth-N} only to interpret $F'$ as a classical derivative,
without affecting the asymptotics.
By Karamata's differentiation (monotone density) theorem for regularly
varying functions (Bingham--Goldie--Teugels \cite[§§1.4--1.7, Thm.~1.7.2]{BGT87};
see also \cite{Grieser2009,Sogge2014}),
if $F(y)=y^\alpha\ell(y)$ with $\alpha>0$ and $\ell\in\RV_0$ is ultimately monotone,
then $F'(y)\sim \alpha y^{\alpha-1}\ell(y)$. Hence
\begin{equation}
F'(y)\sim \frac{d}{2}\,\gamma_d\,\epsilon^{-d/2}\,y^{\frac{d-2}{2}}.
\end{equation}
Since $\rho_{\mathrm{bulk}}(C)=-\frac{d}{dC}\NC(C)=\frac{d}{dy}F(y)$ with $y=\pi-C$,
\eqref{eq:edge-density-asymp} follows and $\rho_{\mathrm{bulk}}\ge 0$.

For the converse, with $\Lambda=(\pi-C)/\epsilon$,
\begin{equation}
N_\Delta(\Lambda)=\NC(\pi-\epsilon\Lambda)\sim A(\epsilon\Lambda)^\alpha
=A\,\epsilon^\alpha\,\Lambda^\alpha\qquad(\Lambda\to\infty),
\end{equation}
so necessarily $\alpha=d/2$ and $\gamma_d=A\,\epsilon^{d/2}$.
\qed
\end{proof}

\begin{corollary}[Practical recovery from bulk edge-variable data]\label{cor:recovery}
Let $y:=\pi-C\to\infty$. If
\begin{equation}
\NC(\pi-y)\sim A\,y^{\alpha}\quad(y\to\infty)
\end{equation}
with ultimate monotonicity (or after smoothing as in Prop.~\ref{prop:smooth-N}), then
\begin{equation}
d=2\alpha,\qquad \gamma_d=A\,\epsilon^{d/2}.
\end{equation}
The log--log slope $\alpha=\lim_{y\to\infty}\frac{d\log \NC(\pi-y)}{d\log y}$ yields $d=2\alpha$, and
\begin{equation}
\gamma_d=\lim_{y\to\infty}\frac{\NC(\pi-y)}{y^{d/2}}\ \epsilon^{d/2}.
\end{equation}
The same holds using the bulk density $\rho_{\mathrm{bulk}}(C)\sim B\,y^{\alpha-1}$, where $d=2\alpha$ and $B=\tfrac{d}{2}\gamma_d\,\epsilon^{-d/2}$.
In practice, given samples of $N_{\mu_C}(C)$ for large negative $C$, consistent estimators are
\begin{equation}
\widehat d(C) := 2\,\frac{d\log N_{\mu_C}(C)}{d\log(\pi-C)},
\qquad
\widehat\gamma_d(C) := \epsilon^{\widehat d(C)/2}\,\frac{N_{\mu_C}(C)}{(\pi-C)^{\widehat d(C)/2}};
\end{equation}
in discrete data, replace the derivative by a local log--log slope.
Clustering contributes only bounded microscopic oscillations and does not affect the leading
bulk asymptotics (see Lemma~\ref{lem:avg-multiplicity}).
\end{corollary}

\subsection{Weyl remainder in the edge variable (bulk regime)}\label{subsec:remainder}

The next propositions quantify how the Weyl remainder and the uniformity in
$\epsilon$ transfer to the edge variable, and extend the main asymptotic to
the smoothed density.

\begin{proposition}[Transfer of remainder]\label{prop:remainder}
If $\NDelta(\Lambda)=\gamma_d\Lambda^{d/2}+O(\Lambda^{(d-1)/2})$ as $\Lambda\to\infty$ (see \cite{Hoermander1968,DuistermaatGuillemin1975,Seeley1978,SafarovVassiliev1997,Ivrii1998}), then, as $C\to-\infty$ (equivalently, $\Lambda=(\pi-C)/\epsilon\to\infty$),
\begin{equation}
\NC(C)=\gamma_d\,\epsilon^{-d/2}(\pi-C)^{d/2}
+O\!\left(\epsilon^{-(d-1)/2}(\pi-C)^{(d-1)/2}\right),
\end{equation}
with an implied constant depending only on $(M,g)$, hence uniform for $\epsilon$ in compact subsets of $(0,\infty)$.
\end{proposition}

\begin{proof}
Apply Lemma~\ref{lem:exact} with $\Lambda=(\pi-C)/\epsilon$ and propagate the remainder. The substitution $\Lambda=(\pi-C)/\epsilon$ and monotonicity preserve both the order and form of the remainder.
\qed
\end{proof}

\begin{proposition}[Uniformity in $\epsilon$ in the bulk regime]\label{prop:uniform-eps}
Fix $0<\epsilon_0<\epsilon_1<\infty$ and let $Y:=(\pi-C)/\epsilon$. Then, as $Y\to\infty$ (equivalently $C\to-\infty$),
\begin{equation}
\sup_{\epsilon\in[\epsilon_0,\epsilon_1]}
\left|\frac{\NC(C)}{\gamma_d\,\epsilon^{-d/2}(\pi-C)^{d/2}}-1\right|\;=\;
\sup_{\epsilon\in[\epsilon_0,\epsilon_1]}
\left|\frac{N_\Delta(Y)}{\gamma_d\,Y^{d/2}}-1\right|\ \longrightarrow\ 0.
\end{equation}
\end{proposition}
\begin{proof}
By Lemma~\ref{lem:exact}, $\NC(C)=N_\Delta\!\big((\pi-C)/\epsilon\big)=N_\Delta(Y)$ with $Y=(\pi-C)/\epsilon$. Hence
\begin{equation}
\sup_{\epsilon\in[\epsilon_0,\epsilon_1]}
\left|\frac{\NC(C)}{\gamma_d\,\epsilon^{-d/2}(\pi-C)^{d/2}}-1\right|
=\left|\frac{N_\Delta(Y)}{\gamma_d\,Y^{d/2}}-1\right|\xrightarrow[Y\to\infty]{}0,
\end{equation}
since $N_\Delta(Y)\sim\gamma_d Y^{d/2}$. 
\qed
\end{proof}

\begin{proposition}[Remainder for bulk density]\label{prop:density-remainder}
If $\NDelta(\Lambda)=\gamma_d\Lambda^{d/2}+O(\Lambda^{(d-1)/2})$ as $\Lambda\to\infty$, then after standard smoothing (as in Prop.~\ref{prop:smooth-N}), the bulk density satisfies
\begin{equation}
\rho_{\mathrm{bulk}}(C)=\tfrac{d}{2}\gamma_d\,\epsilon^{-d/2}(\pi-C)^{(d-2)/2}
+O\!\left(\epsilon^{-(d-2)/2}(\pi-C)^{(d-3)/2}\right),
\end{equation}
as $C\to-\infty$, using the monotone density theorem \cite[Thm.~1.7.2]{BGT87}.
\end{proposition}

\begin{proof}
Apply the monotone density theorem to the remainder estimate from Proposition~\ref{prop:remainder}, using the same smoothing procedure as in Proposition~\ref{prop:smooth-N}.
\qed
\end{proof}

\subsection{Boundary case: two-term bulk expansion}\label{subsec:boundary}

Although Assumption~\ref{ass:smoothness} requires compact manifolds without
boundary, the following extension to the boundary case is standard
\cite{Seeley1978,SafarovVassiliev1997,Ivrii1998} and recorded for completeness.
When $M$ has smooth boundary, Weyl's law acquires a second term proportional to
the boundary area.
The following proposition shows that this two-term structure transfers intact
to the edge variable, with both terms rescaled by the appropriate power of $\epsilon$.

\begin{proposition}[Two-term bulk expansion for manifolds with boundary]\label{prop:boundary}
Let $(M^d,\partial M)$ be compact with smooth boundary and suppose
\begin{equation}
\NDelta(\Lambda)=\gamma_d\,\Lambda^{d/2}+\beta_d\,\Lambda^{(d-1)/2}+o\!\left(\Lambda^{(d-1)/2}\right)
\qquad(\Lambda\to\infty).
\end{equation}
Here and below we consider the standard Dirichlet or Neumann boundary conditions; see \cite[Ch.~1]{Ivrii1998}, \cite[Ch.~8]{SafarovVassiliev1997}.
Then, as $C\to-\infty$ (equivalently, $y:=\pi-C\to\infty$),
\begin{equation}
\NC(C)=\gamma_d\,\epsilon^{-d/2}(\pi-C)^{d/2}
+\beta_d\,\epsilon^{-(d-1)/2}(\pi-C)^{(d-1)/2}
+o\!\left((\pi-C)^{(d-1)/2}\right).
\end{equation}
\end{proposition}

\begin{proof}
Insert $\Lambda=(\pi-C)/\epsilon$ into the two-term Weyl expansion and use Lemma~\ref{lem:exact}.
\qed
\end{proof}

\begin{proposition}[Stable smoothing of the bulk counting]\label{prop:smooth-N}
Let $\varphi\in C_c^\infty(\mathbb{R})$ be a nonnegative mollifier with $\int\varphi=1$, and set $\varphi_h(t):=h^{-1}\varphi(t/h)$. Define the smoothed counting function
\begin{equation}
\NC^{(h)}(C):=(\NC*\varphi_h)(C)=\int \NC(C-t)\,\varphi_h(t)\,dt.
\end{equation}
If $h=h(y)$ satisfies $h(y)=o(y)$ as $y=\pi-C\to\infty$, then
\begin{equation}
\NC^{(h)}(C)\sim \NC(C)\sim \gamma_d\,\epsilon^{-d/2}\,y^{d/2},
\end{equation}
and the derivative $\frac{d}{dC}\NC^{(h)}(C)$ exists with
\begin{equation}
-\frac{d}{dC}\NC^{(h)}(C)=:\rho_{\mathrm{bulk}}^{(h)}(C)\ \sim\ \tfrac{d}{2}\,\gamma_d\,\epsilon^{-d/2}\,y^{\frac{d-2}{2}}.
\end{equation}
\end{proposition}

\begin{proof}
Set $y:=\pi-C$ and $F(y):=\NC(\pi-y)$. By Lemma~\ref{lem:exact} and Weyl's law we have
\begin{equation}
F(y)=N_\Delta(y/\epsilon)\sim A\,y^\alpha,\qquad
A:=\gamma_d\,\epsilon^{-d/2},\ \ \alpha:=\tfrac d2>0,
\end{equation}
hence $F\in \RV_\alpha$ and is ultimately monotone. Let $\varphi\in C_c^\infty(\mathbb{R})$ be nonnegative with $\int\varphi=1$ and write $\varphi_h(t):=h^{-1}\varphi(t/h)$, where $h=h(y)$ satisfies $h(y)=o(y)$ as $y\to\infty$. Then
\begin{equation}
\NC^{(h)}(C)=(\NC*\varphi_h)(C)=(F*\varphi_h)(y)
= \int_{\mathbb{R}} F(y-t)\,\varphi_h(t)\,dt
= \int_{\mathbb{R}} F(y-hu)\,\varphi(u)\,du,
\end{equation}
after the change of variables $t=hu$.

Since $\varphi$ has compact support, say $\mathrm{supp}\,\varphi\subset[-M,M]$, we have $|hu|\le M h=o(y)$. By the Uniform Convergence Theorem for regularly varying functions \cite[Thm.~1.2.1]{BGT87} together with Potter bounds \cite[Thm.~1.5.6]{BGT87}, for every $\varepsilon>0$ there exists $y_0$ such that for all $y\ge y_0$ and $|u|\le M$,
\begin{equation}
\left|\frac{F(y-hu)}{F(y)}-1\right|\le\varepsilon.
\end{equation}
Hence the integrand $\frac{F(y-hu)}{F(y)}\varphi(u)$ is dominated by an integrable bound and converges pointwise (in $u$) to $\varphi(u)$. By dominated convergence,
\begin{equation}
\frac{\NC^{(h)}(C)}{F(y)}=\int \frac{F(y-hu)}{F(y)}\,\varphi(u)\,du \longrightarrow \int \varphi(u)\,du=1,
\end{equation}
so $\NC^{(h)}(C)\sim F(y)\sim \gamma_d\,\epsilon^{-d/2}\,y^{d/2}$ as $y\to\infty$.

For the derivative, convolution with a $C^\infty$ compactly supported kernel yields smoothness and
\begin{equation}
-\frac{d}{dC}\NC^{(h)}(C)=\frac{d}{dy}(F*\varphi_h)(y)=(F'*\varphi_h)(y)
=\int_{\mathbb{R}} F'(y-hu)\,\varphi(u)\,du,
\end{equation}
where the derivative is in the (tempered) distribution sense. Since $F\in \RV_\alpha$ with $\alpha>0$ and is ultimately monotone, the monotone density theorem \cite[Thm.~1.7.2]{BGT87} gives
\begin{equation}
F'(y)\sim \alpha A\,y^{\alpha-1}
=\tfrac{d}{2}\,\gamma_d\,\epsilon^{-d/2}\,y^{\frac{d-2}{2}}
\qquad(y\to\infty),
\end{equation}
so $F'\in \RV_{\alpha-1}$ and is ultimately locally bounded and of one sign. Applying the same Uniform Convergence / dominated--convergence argument to $F'$ yields
\begin{equation}
\frac{d}{dy}(F*\varphi_h)(y)=(F'*\varphi_h)(y)\sim F'(y).
\end{equation}
Recalling $y=\pi-C$, we conclude
\begin{equation}
\rho_{\mathrm{bulk}}^{(h)}(C)=-\frac{d}{dC}\NC^{(h)}(C)
\sim \tfrac{d}{2}\,\gamma_d\,\epsilon^{-d/2}\,y^{\frac{d-2}{2}},
\end{equation}
as claimed.
\qed
\end{proof}

\subsection{Heat-trace and zeta transfer}\label{subsec:heat-zeta}

The affine encoding transfers the heat trace and spectral zeta to the edge
variable by the same rescaling identity that underlies the Weyl transfer above.

\begin{proposition}[Heat and zeta transfer]\label{prop:heat-zeta}
Under the purely affine encoding $C_n=\pi-\epsilon\lambda_n$, define
\begin{equation}
H_{\mathrm{edge}}(s):=\sum_n m_n\,e^{-s(\pi-C_n)},\qquad
\zeta_{\mathrm{edge}}(u):=\sum_n m_n\,(\pi-C_n)^{-u}.
\end{equation}
Then for all $s>0$ and $\Re u>d/2$,
\begin{equation}
H_{\mathrm{edge}}(s)=\Theta_{\Delta}(\epsilon s),\qquad
\zeta_{\mathrm{edge}}(u)=\epsilon^{-u}\,\zeta_{\Delta}(u).
\end{equation}
Consequently, all Seeley--DeWitt heat invariants
\cite{MinakshisundaramPleijel1949,Seeley1967,Gilkey1995} and the pole structure
of $\zeta_{\mathrm{edge}}$ are preserved up to explicit $\epsilon$-scalings;
for constant-curvature spaces this gives $a_{2m}^{\mathrm{edge}}=\epsilon^{m-d/2}a_{2m}$
(see Section~\ref{subsec:const-curv}).
Non-affine perturbations $C_n=\pi-\epsilon\lambda_n+\delta(\lambda_n)$ with
$\delta=o(\lambda)$ break these exact identities; asymptotic recovery is treated
in Section~\ref{sec:stability}.
\end{proposition}

\begin{proof}
Since $\pi-C_n=\epsilon\lambda_n$ for all $n$, both identities follow by direct
substitution: $H_{\mathrm{edge}}(s)=\sum_n m_n e^{-\epsilon s\lambda_n}=\Theta_\Delta(\epsilon s)$
and $\zeta_{\mathrm{edge}}(u)=\epsilon^{-u}\sum_n m_n\lambda_n^{-u}=\epsilon^{-u}\zeta_\Delta(u)$.
\qed
\end{proof}

\section{Uniqueness of the affine encoding}\label{sec:uniqueness}

This section proves the two uniqueness results stated in Theorem~\ref{thm:main}.
The main result is Theorem~\ref{thm:ORV-uniqueness} in
Section~\ref{subsec:uniqueness-orv}: the bulk power-law constraint forces
$\phi\in\mathrm{RV}_1$ within the full O-regularly varying class.
Section~\ref{subsec:poly} recovers the polynomial case $k=1$ as a direct
specialization via de Bruijn's inversion theorem \cite{BGT87,deHaan1996};
the remaining subsections analyze prefactors and slowly varying corrections.

\subsection{Polynomial-type encodings and inverse asymptotics}\label{subsec:poly}

We write $\rho_{\mu_g}(C):=-\frac{d}{dC}N_{\mu_g}(C)$
in the Stieltjes/distributional sense (or after a standard smoothing).

The following specializes the rigidity to the polynomial class $\mathrm{RV}_k$;
the self-contained proof via de Bruijn inversion makes the exponent formula
explicit. The general O-RV statement (Theorem~\ref{thm:ORV-uniqueness}) is
established in Section~\ref{subsec:uniqueness-orv} below.

\begin{theorem}[Polynomial specialization]\label{thm:unique}
Let $g(\lambda)\sim a-b\,\lambda^{k}L(\lambda)$ as $\lambda\to\infty$ with
$a\in\mathbb{R}$, $b>0$, $k>0$, and $L\in \RV_0$ (slowly varying). Assume $g$ is
ultimately strictly decreasing so that $a$ is the upper edge of the pushforward
measure $\mu_g$ of $\mu_\Delta$ under $g$ (by strict monotonicity of $g$; cf. Lemma~\ref{lem:mono-encoding} for the affine case). Put $\phi(\lambda):=a-g(\lambda)$ and
$x:=a-C\downarrow 0$. Let $\Lambda$ be the asymptotic inverse of $\phi$.
Then, as $C\uparrow a$,
\begin{equation}
N_{\mu_g}(a-x)=N_\Delta(\Lambda(x))\in \RV_{\,d/(2k)},\qquad
\rho_{\mu_g}(C)\in \RV_{\,d/(2k)-1}.
\end{equation}
In particular, to reproduce the \emph{edge-variable} density exponent $(d-2)/2$ of
\eqref{eq:edge-density-asymp} one must have $k=1$, i.e.\ $g(\lambda)=a-b\lambda(1+o(1))$.
\end{theorem}

\begin{proof}
We have $\phi(\lambda)=a-g(\lambda)\sim b\,\lambda^{k}L(\lambda)$, so $\phi\in\RV_k$
and is ultimately increasing. By inversion theory for regularly varying functions
(de Bruijn conjugacy; \cite[§1.5, Thm.~1.5.12]{BGT87}; see also \cite{deHaan1996}), there exists an asymptotic
inverse $\Lambda\in\RV_{1/k}$ with
\begin{equation}\label{eq:debruijn}
\Lambda(x)\sim \Big(\frac{x}{b}\Big)^{1/k}\,L^\ast(x)\qquad(x\downarrow 0),
\end{equation}
where $L^\ast\in\RV_0$ is characterized by
\begin{equation}
\big(L^\ast(x)\big)^{k}\,
L\!\Big(\Big(\frac{x}{b}\Big)^{\!1/k} L^\ast(x)\Big)\ \longrightarrow\ 1
\qquad(x\downarrow 0).
\end{equation}
Monotonicity yields the exact composition
$N_{\mu_g}(a-x)=N_\Delta(\Lambda(x))$. Applying Weyl's law gives
\begin{equation}
N_{\mu_g}(a-x)\sim \gamma_d\,\Lambda(x)^{d/2}\in \RV_{\,d/(2k)}\qquad(x\downarrow 0).
\end{equation}
Let $F(x):=N_{\mu_g}(a-x)$. By the monotone density theorem
(\cite[§1.7, Thm.~1.7.2]{BGT87}), $F'(x)\in \RV_{\,d/(2k)-1}$. Since
$\rho_{\mu_g}(C)=-\frac{d}{dC}N_{\mu_g}(C)=\frac{d}{dx}F(x)$ with $x=a-C$,
the edge density exponent is $d/(2k)-1$.
The nontrivial content is the inverse transfer $\phi\in\RV_k\Rightarrow
\Lambda\in\RV_{1/k}$ via de Bruijn conjugacy \eqref{eq:debruijn} and the
monotone density theorem; the exponent matching is then an unavoidable
consequence: $d/(2k)-1=(d-2)/2$ forces $k=1$.
Slowly varying factors $L,L^\ast$ affect the prefactor but not the exponent.
\qed
\end{proof}

\begin{corollary}[Nonlinear distortion of the edge-variable exponent]\label{cor:nonlinear}
Under the assumptions of Theorem~\ref{thm:unique}, if $k\neq1$ then the edge-variable density exponent equals $d/(2k)-1\neq (d-2)/2$. For instance, for $k=2$ one has
\begin{equation}
\rho_{\mu_g}(C)\in \mathrm{RV}_{\,d/4-1},
\end{equation}
which mismatches the geometric edge-variable exponent for any $d\ge 1$.
\end{corollary}

\begin{corollary}[Affine optimality for the edge-variable exponent]\label{cor:affine-optimality}
Among encodings $g(\lambda)=a-b\lambda^{k}L(\lambda)$ with $b>0$, $k>0$, $L\in\RV_0$,
the only ones reproducing both the correct edge-variable exponent $(d-2)/2$ and a finite,
nonzero edge-variable prefactor are those with $k=1$. If moreover $L(\lambda)\to1$, the
prefactor matches $\tfrac{d}{2}\gamma_d\,b^{-d/2}$; any $L\not\to1$ induces
a slowly varying multiplicative distortion.
\end{corollary}

\begin{proof}
Immediate from Theorem~\ref{thm:unique} and the monotone density theorem \cite[§1.7]{BGT87}.
\qed
\end{proof}

\subsection{Recovering $k$ from edge-variable data}\label{subsec:k-equals-one}

By Theorem~\ref{thm:unique}, the exponent formula $N_{\mu_g}(a-x)\in\RV_{d/(2k)}$
immediately yields a practical estimator for $k$.

\begin{corollary}[Estimating $k$ from edge-variable data]\label{cor:estimate-k}
Suppose near the upper edge $C\uparrow a$ the counting behaves as
\begin{equation}
N_{\mu_g}(C)\sim A\,(a-C)^{\alpha}\quad\text{with }\alpha>0,
\end{equation}
or, equivalently, the edge-variable density behaves as
\begin{equation}
\rho_{\mu_g}(C)\sim B\,(a-C)^{\beta}\quad\text{with }\beta>-1.
\end{equation}
Then $\alpha=d/(2k)$ and $\beta=d/(2k)-1$, so
\begin{equation}
k=\frac{d}{2\alpha}=\frac{d}{2(\beta+1)}.
\end{equation}
Thus the edge-variable log--log slope (of either $N$ or $\rho$) recovers $k$; matching the geometric value requires $k=1$.
\end{corollary}

\subsection{Slow variation and prefactors}\label{subsec:slow-var}

While $k=1$ is forced by the exponent constraint, the slowly varying factor
$L\in\mathrm{RV}_0$ in the affine encoding $g(\lambda)=a-b\lambda L(\lambda)$
can still affect the prefactor via the de Bruijn conjugate $L^\ast$.
This subsection quantifies that effect and shows the prefactor reduces to
the canonical form when $L\to1$.

\begin{proposition}[Slow variation does not change edge-variable exponents]\label{prop:slow-var-exponents}
Let $g(\lambda)\sim a-b\,\lambda^{k}L(\lambda)$ with $b>0$, $k>0$, and $L\in\RV_0$, and assume $g$ is ultimately strictly decreasing. Put $\phi(\lambda):=a-g(\lambda)\sim b\lambda^{k}L(\lambda)$ and let $\Lambda$ be an asymptotic inverse of $\phi$. Then, with $x:=a-C$,
\begin{equation}
N_{\mu_g}(a-x)=N_\Delta(\Lambda(x))\in\RV_{\,d/(2k)},
\qquad
\rho_{\mu_g}(a-x)\in\RV_{\,d/(2k)-1}.
\end{equation}
In particular, $L\in\RV_0$ can modify only the \emph{prefactor} by a slowly varying multiplicative factor; the exponents $d/(2k)$ and $d/(2k)-1$ are unaffected.
\end{proposition}

\begin{proof}
By de Bruijn inversion \cite[Thm.~1.5.12]{BGT87}, there exists $L^\ast\in\RV_0$ such that
\begin{equation}
\Lambda(x)\sim \Big(\frac{x}{b}\Big)^{1/k}L^\ast(x)
\quad\text{with}\quad
\big(L^\ast(x)\big)^{k}\,
L\!\Big(\Big(\tfrac{x}{b}\Big)^{1/k} L^\ast(x)\Big)\to 1.
\end{equation}
Then Weyl's law gives
\begin{equation}
N_{\mu_g}(a-x)\sim \gamma_d\,\Lambda(x)^{d/2}
\sim \gamma_d\,b^{-d/(2k)}\,x^{d/(2k)}\,(L^\ast(x))^{d/2}\in\RV_{\,d/(2k)},
\end{equation}
and the monotone density theorem yields the claim for $\rho_{\mu_g}$.
\qed
\end{proof}

\begin{corollary}[Explicit prefactor with de Bruijn conjugate]\label{cor:prefactor-Lstar}
Under the assumptions of Proposition~\ref{prop:slow-var-exponents},
\begin{align}
N_{\mu_g}(a-x)&\sim \gamma_d\,b^{-d/(2k)}\,x^{d/(2k)}\,(L^\ast(x))^{d/2},\\
\rho_{\mu_g}(a-x)&\sim \tfrac{d}{2k}\,\gamma_d\,b^{-d/(2k)}\,x^{\frac{d}{2k}-1}\,(L^\ast(x))^{d/2},
\end{align}
where $L^\ast\in\RV_0$ is the de Bruijn conjugate of $L$ associated to $\phi(\lambda)=b\lambda^{k}L(\lambda)$.
\end{corollary}

\begin{remark}[Specializations for $k=1$]\label{rem:slow-k1}
When $k=1$ (the asymptotically affine case required by Theorem~\ref{thm:unique}):
\begin{enumerate}[leftmargin=2em,itemsep=2pt]
\item If $L(\lambda)\to \ell_\infty\in(0,\infty)$, then $L^\ast(x)\to \ell_\infty^{-1}$ and
\begin{align}
N_{\mu_g}(a-x)&\sim \gamma_d\, (b\ell_\infty)^{-d/2}\,x^{d/2},\\
\rho_{\mu_g}(a-x)&\sim \tfrac{d}{2}\,\gamma_d\, (b\ell_\infty)^{-d/2}\,x^{\frac{d-2}{2}}.
\end{align}
Thus a constant slow factor merely rescales the prefactor by $(\ell_\infty)^{-d/2}$.
\item If $L(\lambda)=(\log\lambda)^{\alpha}(\log\log\lambda)^{\beta}$ for large $\lambda$,
then $L^\ast(x)$ is again slowly varying and behaves (in the corresponding asymptotic regime)
like $(\log x)^{-\alpha}(\log\log x)^{-\beta}$; hence $N_{\mu_g}$ and $\rho_{\mu_g}$ pick up
a multiplicative modulation $(L^\ast(x))^{d/2}$ while their power exponents stay unchanged.
\end{enumerate}
\end{remark}

\begin{corollary}[Affine normalization and matching of prefactors]\label{cor:affine-normalization}
In the affine case $g(\lambda)=a-\epsilon\,\lambda\,L(\lambda)$ with $L\to 1$, the de Bruijn factor satisfies $L^\ast(x)\to 1$ and the edge-variable prefactors reduce to the canonical values
\begin{align}
N_{\mu_g}(a-x)&\sim \gamma_d\,\epsilon^{-d/2}\,x^{d/2},\\
\rho_{\mu_g}(a-x)&\sim \tfrac{d}{2}\,\gamma_d\,\epsilon^{-d/2}\,x^{\frac{d-2}{2}}.
\end{align}
More generally, any $L\in\RV_0$ only changes these prefactors by the slowly varying factor $(L^\ast(x))^{d/2}$.
\end{corollary}

\subsection{Uniqueness of the affine scaling from the bulk exponent}\label{subsec:uniqueness-orv}

Let $g(\lambda)=a-\phi(\lambda)$ with $\phi(\lambda)>0$ ultimately increasing and $\phi(\lambda)\to\infty$. Write $x:=a-C\downarrow0$ and let $\mu_g$ be the pushforward of $\mu_\Delta$ by $g$.

The following is the main rigidity theorem of this section;
Theorem~\ref{thm:unique} above is its polynomial specialization.

\begin{theorem}[O-RV rigidity: main theorem]\label{thm:ORV-uniqueness}
Assume $\phi$ is \(O\)-regularly varying at $+\infty$ with finite Matuszewska indices $0<\alpha_\phi^-\le \alpha_\phi^+<\infty$ (Bingham--Goldie--Teugels~\cite[§2.1]{BGT87}), and is ultimately increasing. Let $\Lambda$ be an asymptotic inverse of $\phi$ near $0$. Then:
\begin{enumerate}[label=\textup{(\roman*)}, itemsep=2pt]
\item $\Lambda$ is \(O\)-regularly varying at $0+$ with indices $1/\alpha_\phi^+\le \mathrm{ind}(\Lambda)\le 1/\alpha_\phi^-$.
\item Consequently,
\begin{align}
N_{\mu_g}(a-x)&=N_\Delta(\Lambda(x))\in O\text{-}\RV_{\,\left[d/(2\alpha_\phi^+),\,d/(2\alpha_\phi^-)\right]},\\
\rho_{\mu_g}(a-x)&\in O\text{-}\RV_{\,\left[d/(2\alpha_\phi^+)-1,\,d/(2\alpha_\phi^-)-1\right]}.
\end{align}
\item If, in fact, a genuine power law holds at the edge,
\begin{equation}
N_{\mu_g}(a-x)\sim A\,x^{d/2}\quad\text{or}\quad \rho_{\mu_g}(a-x)\sim B\,x^{(d-2)/2},
\end{equation}
then necessarily $\alpha_\phi^-=\alpha_\phi^+=1$. Hence $\phi\in \RV_1$ and
\begin{equation}
\phi(\lambda)=b\,\lambda\,L(\lambda),\qquad b>0,\quad L\in\RV_0,
\end{equation}
so $g(\lambda)=a-b\lambda L(\lambda)$ is asymptotically affine (up to slow variation). In particular, the case $k=1$ of Theorem~\ref{thm:unique} is \emph{necessary and sufficient} at the level of one-sided bulk power exponents.
\end{enumerate}
\end{theorem}

\begin{proof}
(i)--(ii) follow from standard inversion properties of \(O\)-regularly varying functions and bounding the inverse via the upper/lower Matuszewska indices \cite[§§2.1--2.4]{BGT87}. Composing with Weyl's law $N_\Delta(\Lambda)\sim \gamma_d\,\Lambda^{d/2}$ transfers the index interval. (iii) A single sharp power exponent for $N_{\mu_g}$ (or for $\rho_{\mu_g}$ via the monotone density theorem) forces the upper and lower indices of $\Lambda$ to coincide, hence those of $\phi$ coincide and equal $1$, i.e. $\phi\in\RV_1$. Then $\phi(\lambda)=b\lambda L(\lambda)$ with $L\in\RV_0$ by Karamata's representation \cite[Thm.~1.3.1]{BGT87}, which yields the stated form of $g$.
\qed
\end{proof}

\begin{remark}[Rapid/subliminal distortions cannot reproduce the bulk power]
If $\phi$ is faster than any power (e.g.\ $\phi(\lambda)=e^\lambda$), then $\Lambda(x)\asymp \log(1/x)$ and
$N_{\mu_g}(a-x)\asymp (\log(1/x))^{d/2}$, $\rho_{\mu_g}(a-x)\asymp x^{-1}(\log(1/x))^{d/2-1}$, which is not a power law. If $\phi(\lambda)=\lambda^k L(\lambda)$ with $k\neq1$, Theorem~\ref{thm:unique} shows the exponent mismatch $d/(2k)-1\neq (d-2)/2$.
\end{remark}

\subsection{Asymptotic spectral classes and encoding-induced morphisms}\label{subsec:asymptotic-classes}

\begin{definition}[Asymptotic spectral equivalence class]
Let $\mu_1,\mu_2$ be positive spectral measures with ultimately monotone counting
functions $N_1,N_2$. For $\beta>0$ write
\begin{equation}
\mu_1\sim_{\mathrm{as},\beta}\mu_2
\quad\Longleftrightarrow\quad
N_1\in\RV_\beta\ \text{and}\ N_2\in\RV_\beta.
\end{equation}
The corresponding class is denoted $[\mu]_{\mathrm{as},\beta}$.
Whenever $N_\mu\in\RV_\beta$, define the asymptotic spectral dimension
\begin{equation}
d_{\mathrm{as}}(\mu):=2\beta.
\end{equation}
\end{definition}

\begin{proposition}[Encoding-induced dimension scaling]\label{prop:class-morphism}
Assume $N_\Delta(\Lambda)\sim \gamma_d\Lambda^{d/2}$ and let
$g(\lambda)=a-\phi(\lambda)$ be ultimately strictly decreasing with
$\phi\in\RV_k$, $k>0$. Then
\begin{equation}
N_{\mu_g}(a-x)\in\RV_{d/(2k)},
\end{equation}
so $g$ induces a map between asymptotic classes
\begin{equation}
T_g:\ [\mu_\Delta]_{\mathrm{as},\,d/2}\longrightarrow [\mu_g]_{\mathrm{as},\,d/(2k)},
\end{equation}
and therefore acts on asymptotic spectral dimension by
\begin{equation}
d_{\mathrm{as}}\longmapsto \frac{1}{k}\,d_{\mathrm{as}}.
\end{equation}
\end{proposition}

\begin{proof}
This is exactly Theorem~\ref{thm:unique}: if $\phi\in\RV_k$ then
$N_{\mu_g}(a-x)=N_\Delta(\Lambda(x))\in\RV_{d/(2k)}$, where $\Lambda$ is an
asymptotic inverse of $\phi$.
\qed
\end{proof}

\begin{corollary}[Dimension-preserving morphisms]\label{cor:dim-preserving-morphism}
Under the hypotheses of Theorem~\ref{thm:ORV-uniqueness}:
\begin{enumerate}[label=\textup{(\roman*)},itemsep=2pt]
\item $g(\lambda)=a-\phi(\lambda)$ preserves asymptotic spectral dimension
if and only if $N_{\mu_g}(a-x)\in\RV_{d/2}$ (equivalently
$\rho_{\mu_g}\in\RV_{(d-2)/2}$);
\item if a genuine edge power law holds, namely
\begin{equation}
N_{\mu_g}(a-x)\sim A\,x^{d/2}
\quad\text{or}\quad
\rho_{\mu_g}(a-x)\sim B\,x^{(d-2)/2},
\end{equation}
then necessarily $\phi\in\RV_1$;
\item conversely, $\phi\in\RV_1$ implies
$N_{\mu_g}(a-x)\in\RV_{d/2}$ and hence dimension preservation.
\end{enumerate}
If in addition $\phi(\lambda)=\epsilon\lambda L(\lambda)$ with $L(\lambda)\to1$,
then
\begin{equation}
N_{\mu_g}(a-x)\sim \gamma_d\,\epsilon^{-d/2}x^{d/2},
\end{equation}
and the normalization is asymptotically strict affine:
$g(\lambda)=a-\epsilon\lambda+o(\lambda)$.
\end{corollary}

\begin{proof}
The first statement is the definition of $d_{\mathrm{as}}=2\beta$.
The implication in \textup{(ii)} is Theorem~\ref{thm:ORV-uniqueness}(iii).
For \textup{(iii)}, if $\phi\in\RV_1$, then its asymptotic inverse satisfies
$\Lambda\in\RV_1$, and therefore
$N_{\mu_g}(a-x)=N_\Delta(\Lambda(x))\in\RV_{d/2}$ by Weyl's law and composition
of regularly varying functions.
The normalization statement follows from Corollary~\ref{cor:affine-normalization}.
\qed
\end{proof}

\begin{remark}[Hierarchy of preservation]
\begin{enumerate}[label=\textup{(\roman*)},itemsep=2pt]
\item preserving Weyl exponent (equivalently asymptotic dimension) is exactly $\phi\in\RV_1$;
\item preserving the normalized Weyl constant requires the first-order normalization
$\phi(\lambda)=\epsilon\lambda(1+o(1))$;
\item this is the strict affine regime at leading order.
\end{enumerate}
\end{remark}

\section{Stability under perturbations}\label{sec:stability}

The uniqueness results of Section~\ref{sec:uniqueness} concern exact encodings.
Here we show that the geometric bulk exponents are stable when the affine encoding
is perturbed by $\delta(\lambda)=o(\lambda)$: the counting function and bulk
density retain the same power-law exponents, with error rates controlled by the
slowly varying factor $\eta=\delta/\lambda\in\mathrm{RV}_0$ \cite{BGT87}.
This stability is relevant whenever the encoding is known only approximately,
for instance when $\epsilon$ is estimated from data or when the encoding
is subject to a systematic drift.

\subsection{Setup and inversion error}\label{subsec:setup-stab}

The starting point is an explicit inversion estimate: given the perturbed
encoding $C(\lambda)=\pi-\epsilon\lambda+\delta(\lambda)$, we bound the
deviation of the inverse $\Lambda(C)$ from the unperturbed value $y/\epsilon$.

\begin{lemma}[Explicit inversion error under tame perturbations]\label{lem:inverse-error}
Assume $C(\lambda)=\pi-\epsilon\lambda+\delta(\lambda)$ is eventually strictly decreasing, with
\begin{equation}
|\delta(\lambda)|\le \eta(\lambda)\,\lambda,\qquad \eta\in \mathrm{RV}_0,\ \eta(\lambda)\to0.
\end{equation}
Let $y:=\pi-C$ and denote by $\Lambda(C)$ the inverse of $\lambda\mapsto C(\lambda)$ for large $\lambda$ (equivalently, for large $y$). Then, as $C\to-\infty$ (i.e. $y\to+\infty$),
\begin{equation}
\Lambda(C)=\frac{y}{\epsilon}\Big(1+O\big(\eta(y/\epsilon)\big)\Big).
\end{equation}
Consequently,
\begin{equation}
\NC(C)=\NDelta(\Lambda(C))=\gamma_d\,\epsilon^{-d/2}\,y^{d/2}\Big(1+O\big(\eta(y/\epsilon)\big)\Big),
\end{equation}
and the same relative error carries to $\rho_{\mathrm{bulk}}(C)$ after (Stieltjes) differentiation.
\end{lemma}

\begin{proof}
Set $y:=\pi-C$ and fix $\Lambda_0$ so large that for all $\lambda\ge \Lambda_0$ one has $\eta(\lambda)\le \epsilon/2$ and $C(\lambda)$ is strictly decreasing; consequently the inverse $\Lambda(C)$ is well defined for all $y\ge Y_0:=\pi-C(\Lambda_0)$.

\emph{Step 1: two--sided inversion bounds.}
From $y=\epsilon\lambda-\delta(\lambda)$ and $|\delta(\lambda)|\le \eta(\lambda)\lambda$ we get
\begin{equation}
\lambda\big(\epsilon-\eta(\lambda)\big)\ \le\ y\ \le\ \lambda\big(\epsilon+\eta(\lambda)\big),
\end{equation}
hence, for all large $y$ (so that $\eta(\lambda)<\epsilon/2$),
\begin{equation}\label{eq:two-sided-lambda}
\frac{y}{\epsilon+\eta(\lambda)}\ \le\ \lambda\ \le\ \frac{y}{\epsilon-\eta(\lambda)}.
\end{equation}
Dividing by $y/\epsilon$ and using $(1\pm t)^{-1}=1\mp t+O(t^2)$ yields
\begin{equation}
\frac{\lambda}{y/\epsilon}\ \in\ \left[\,\frac{\epsilon}{\epsilon+\eta(\lambda)}\,,\,\frac{\epsilon}{\epsilon-\eta(\lambda)}\,\right]
=\left[\,1-\frac{\eta(\lambda)}{\epsilon}+O\!\big(\eta(\lambda)^2\big),\ 1+\frac{\eta(\lambda)}{\epsilon}+O\!\big(\eta(\lambda)^2\big)\right].
\end{equation}
Therefore
\begin{equation}\label{eq:lambda-y-over-eps}
\lambda=\frac{y}{\epsilon}\Big(1+O\big(\eta(\lambda)\big)\Big)\qquad(y\to\infty),
\end{equation}
and in particular $\lambda/(y/\epsilon)\to 1$ as $y\to\infty$.

\emph{Step 2: replacing $\eta(\lambda)$ by $\eta(y/\epsilon)$.}
Since $\eta\in\RV_0$ (slowly varying), the Uniform Convergence Theorem (Bingham--Goldie--Teugels \cite[Thm.~1.2.1]{BGT87}) gives
\begin{equation}
\frac{\eta(\lambda)}{\eta(y/\epsilon)}\ \longrightarrow\ 1\qquad\text{whenever}\quad \frac{\lambda}{y/\epsilon}\ \longrightarrow\ 1.
\end{equation}
Using \eqref{eq:lambda-y-over-eps}, this applies here, so \(\eta(\lambda)=\eta(y/\epsilon)\,(1+o(1))\). Thus
\begin{equation}\label{eq:lambda-main}
\Lambda(C)=\lambda=\frac{y}{\epsilon}\Big(1+O\big(\eta(y/\epsilon)\big)\Big)\qquad(y\to\infty),
\end{equation}
which is the first claimed estimate.

\emph{Step 3: asymptotics for $N_{\mu_C}$.}
By monotonicity of $C(\cdot)$ we have the exact identity $N_{\mu_C}(C)=\#\{\lambda_\ell\le \Lambda(C)\}=N_\Delta(\Lambda(C))$. (Cf. Lemma~\ref{lem:mono-encoding}.) Weyl's law yields
\begin{equation}
N_\Delta(\Lambda)=\gamma_d\,\Lambda^{d/2}\,(1+o(1))\qquad(\Lambda\to\infty).
\end{equation}
Insert \eqref{eq:lambda-main}: write $\Lambda=(y/\epsilon)\,(1+\theta(y))$ with $\theta(y)=O(\eta(y/\epsilon))$. Then
\begin{equation}
\Lambda^{d/2}=\Big(\tfrac{y}{\epsilon}\Big)^{\!d/2}\,(1+\theta(y))^{d/2}
=\Big(\tfrac{y}{\epsilon}\Big)^{\!d/2}\Big(1+O\big(\eta(y/\epsilon)\big)\Big),
\end{equation}
and hence
\begin{equation}\label{eq:NC-asymp}
N_{\mu_C}(C)=\gamma_d\,\epsilon^{-d/2}\,y^{d/2}\Big(1+O\big(\eta(y/\epsilon)\big)\Big)\qquad(y\to\infty).
\end{equation}

\emph{Step 4: bulk density.}
Define $F(y):=N_{\mu_C}(\pi-y)$, so $F$ is nondecreasing and $F(y)\sim A\,y^{\alpha}L(y)$ with
\begin{equation}
A:=\gamma_d\,\epsilon^{-d/2},\quad \alpha:=\tfrac d2>0,\quad L(y):=1+O\big(\eta(y/\epsilon)\big).
\end{equation}
Because $\eta\in\RV_0$, the factor $L$ is slowly varying (and $L(y)\to 1$). By the monotone density (Karamata differentiation) theorem for regularly varying functions \cite[§§1.4--1.7, Thm.~1.7.2]{BGT87}, one has
\begin{equation}
F'(y)\ \sim\ \alpha\,A\,y^{\alpha-1}\,L(y)
=\tfrac{d}{2}\,\gamma_d\,\epsilon^{-d/2}\,y^{\frac{d-2}{2}}\Big(1+O\big(\eta(y/\epsilon)\big)\Big),
\end{equation}
where the derivative is taken in the (Stieltjes/distributional) sense; equivalently, the same asymptotic holds for any standard smooth regularization of the density. Since $\rho_{\mathrm{bulk}}(C)=-\frac{d}{dC}N_{\mu_C}(C)=\frac{d}{dy}F(y)$ with $y=\pi-C$, this gives the stated estimate for $\rho_{\mathrm{bulk}}(C)$.

Combining the four steps completes the proof.
\qed
\end{proof}

\subsection{Preservation of edge-variable exponents}\label{subsec:preserve}

Using the inversion bound from Lemma~\ref{lem:inverse-error}, we now state the
main stability result: all three quantities $\Lambda(C)$, $N_{\mu_C}(C)$, and
$\rho_{\mathrm{bulk}}(C)$ have the same asymptotics as in the unperturbed case,
up to a multiplicative factor $1+O(\eta(y/\epsilon))$.

\begin{proposition}[Stability under small perturbations: bulk formulation]\label{prop:stability}
Let
\begin{equation}
C(\lambda)=\pi-\epsilon\lambda+\delta(\lambda),\qquad \epsilon>0,
\end{equation}
where $\delta(\lambda)=\lambda\,\eta(\lambda)$ with $\eta\in \mathrm{RV}_0$ and $\eta(\lambda)\to 0$ as $\lambda\to\infty$.
Assume moreover that $\delta$ is $C^1$ eventually, $\delta'\in \mathrm{RV}_0$, and $\eta$ is ultimately monotone (equivalently, of bounded variation on a tail).
Set $y:=\pi-C$ and let $\Lambda(C)$ be the inverse of $\lambda\mapsto C(\lambda)$.
Then, as $C\to-\infty$ (equivalently $y\to+\infty$),
\begin{align}
\Lambda(C) &= \frac{y}{\epsilon}\Big(1 + O\!\big(\eta(y/\epsilon)\big)\Big),\\
N_{\mu_C}(C) &= \gamma_d\,\epsilon^{-d/2}\,y^{d/2}\Big(1 + O\!\big(\eta(y/\epsilon)\big)\Big),\\
\rho_{\mathrm{bulk}}(C) &= \tfrac{d}{2}\,\gamma_d\,\epsilon^{-d/2}\,y^{\frac{d-2}{2}}\Big(1 + O\!\big(\eta(y/\epsilon)\big)\Big).
\end{align}
hence the edge-variable (bulk) exponents are preserved and the prefactors are perturbed by a multiplicative $(1+O(\eta(y/\epsilon)))$ factor (in particular, by $(1+o(1))$). The $O(\cdot)$ terms are uniform for $\epsilon\in[\epsilon_0,\epsilon_1]\subset(0,\infty)$.
\end{proposition}

\begin{proof}
Write $y=\pi-C(\lambda)=\epsilon\lambda-\delta(\lambda)=\epsilon\lambda\,[1-\eta(\lambda)/\epsilon]$.
Thus
\begin{equation}
\lambda=\frac{y}{\epsilon}\,\Big[1-\eta(\lambda)/\epsilon\Big]^{-1}
=\frac{y}{\epsilon}\,\Big(1+\tfrac{1}{\epsilon}\eta(\lambda)+O(\eta(\lambda)^2)\Big).
\end{equation}
By the uniform convergence theorem for slowly varying functions \cite[§1.8]{BGT87}, replace $\eta(\lambda)$ by $\eta(y/\epsilon)$ when $\lambda\sim y/\epsilon$, yielding
\begin{equation}
\Lambda(C)=\frac{y}{\epsilon}\Big(1+O(\eta(y/\epsilon))\Big),\qquad y\to\infty.
\end{equation}
Now apply Weyl's law $N_\Delta(\Lambda)\sim \gamma_d\,\Lambda^{d/2}$ (valid as $\Lambda\to\infty$) and stability of regular variation under composition \cite[§1.7]{BGT87} to get
\begin{equation}
N_{\mu_C}(C)=N_\Delta(\Lambda(C))
=\gamma_d\,\epsilon^{-d/2}\,y^{d/2}\Big(1+O(\eta(y/\epsilon))\Big).
\end{equation}
Set $F(y):=N_{\mu_C}(\pi-y)$. Then $F\in \mathrm{RV}_{d/2}$ and is ultimately monotone; by Karamata's monotone density theorem \cite[Thm.~1.7.2]{BGT87},
\begin{equation}
F'(y)\sim \tfrac{d}{2}\,\gamma_d\,\epsilon^{-d/2}\,y^{\frac{d-2}{2}}\Big(1+O(\eta(y/\epsilon))\Big).
\end{equation}
Since $\rho_{\mathrm{bulk}}(C)=-\frac{d}{dC}N_{\mu_C}(C)=\frac{d}{dy}F(y)$ with $y=\pi-C$, the stated asymptotics for $\rho_{\mathrm{bulk}}$ follow.
\qed
\end{proof}

\begin{lemma}[A convenient growth condition on $\delta$]\label{lem:delta-growth}
Assume $\delta(\lambda)=o(\lambda)$, $\delta$ is ultimately monotone, and there exists
a slowly varying $\ell\in\RV_0$ with
\begin{equation}
|\delta'(\lambda)|\le \frac{\ell(\lambda)}{\lambda}\quad\text{for large }\lambda.
\end{equation}
Then the conclusions of Lemma~\ref{lem:inverse-error} and Proposition~\ref{prop:stability}
hold (in the bulk formulation above) with $\eta=\ell$.
\end{lemma}
\begin{proof}
Integrating $|\delta'(\lambda)|\le \ell(\lambda)/\lambda$ yields $\delta(\lambda)=o(\lambda)$
with explicit control; then apply Lemma~\ref{lem:inverse-error} and the proof of
Proposition~\ref{prop:stability}, using the uniform convergence theorem for slowly varying functions \cite[§1.8]{BGT87}.
\qed
\end{proof}

\subsection{Examples of admissible perturbations}\label{subsec:examples-stab}

We illustrate the stability result with explicit families of perturbations,
computing the slowly varying factor $\eta$ and the resulting bulk error profile
in each case.

\begin{proposition}[Admissible perturbations and bulk error profiles]\label{prop:examples-stab}
Let $C(\lambda)=\pi-\epsilon\lambda+\delta(\lambda)$ be ultimately strictly decreasing and put $y:=\pi-C\to\infty$. The following choices of $\delta$ are admissible:
\begin{enumerate}[label=\textup{(\alph*)}, itemsep=3pt]
\item Logarithmic slope distortions. For $\alpha>0$ set
\begin{equation}
\delta(\lambda)=\frac{\lambda}{(\log(e\lambda))^\alpha},\qquad
\eta(\lambda)=\frac{1}{(\log(e\lambda))^\alpha}\in\RV_0.
\end{equation}
Then
\begin{equation}
\Lambda(C)=\frac{y}{\epsilon}\Big(1+O((\log(y/\epsilon))^{-\alpha})\Big),
\end{equation}
\begin{align}
N_{\mu_C}(C)&=\gamma_d\epsilon^{-d/2}y^{d/2}\Big(1+O((\log(y/\epsilon))^{-\alpha})\Big),\\
\rho_{\mathrm{bulk}}(C)&=\tfrac{d}{2}\gamma_d\epsilon^{-d/2}y^{\frac{d-2}{2}}\Big(1+O((\log(y/\epsilon))^{-\alpha})\Big).
\end{align}

\item Iterated logarithms. For $\alpha>0$, $\beta\in\mathbb{R}$ let
\begin{equation}
\delta(\lambda)=\frac{\lambda}{(\log(e\lambda))^\alpha(\log\log(e^e\lambda))^\beta},\quad
\eta(\lambda)=\frac{1}{(\log(e\lambda))^\alpha(\log\log(e^e\lambda))^\beta}\in\RV_0.
\end{equation}
Then all conclusions in (a) hold with
\begin{equation}
O(\eta(y/\epsilon))=O\!\Big((\log(y/\epsilon))^{-\alpha}(\log\log(y/\epsilon))^{-\beta}\Big).
\end{equation}

\item General slowly varying multiplicative distortions.
Let $L\in\RV_0$ be ultimately monotone with $L(\lambda)\to0$, and put $\delta(\lambda)=\lambda L(\lambda)$ (so $\eta=L$). Then
\begin{align}
\Lambda(C)&=\frac{y}{\epsilon}\big(1+O(L(y/\epsilon))\big),\\
N_{\mu_C}(C)&=\gamma_d\epsilon^{-d/2}y^{d/2}\big(1+O(L(y/\epsilon))\big),
\end{align}
and the same relative error holds for $\rho_{\mathrm{bulk}}(C)$.

\item Bounded additive offsets. If $\delta(\lambda)=O(1)$ and $C$ is ultimately strictly decreasing, then
\begin{align}
\Lambda(C)&=\frac{y}{\epsilon}+O(1),\\
N_{\mu_C}(C)&=\gamma_d\epsilon^{-d/2}y^{d/2}\Big(1+O(y^{-1})\Big),
\end{align}
\begin{equation}
\rho_{\mathrm{bulk}}(C)=\tfrac{d}{2}\gamma_d\epsilon^{-d/2}y^{\frac{d-2}{2}}\Big(1+O(y^{-1})\Big).
\end{equation}

\item Sublogarithmic additive terms. For $0<\beta<1$, let
\begin{equation}
\delta(\lambda)=(\log(e\lambda))^{\beta}.
\end{equation}
Then $\delta(\lambda)=o(\lambda)$ and
$\delta'(\lambda)=\beta(\log(e\lambda))^{\beta-1}\lambda^{-1}\le \ell(\lambda)/\lambda$ with $\ell(\lambda)=\beta(\log(e\lambda))^{\beta-1}\in\RV_0$, $\ell(\lambda)\to0$. Hence
\begin{equation}
\Lambda(C)=\frac{y}{\epsilon}\Big(1+O((\log(y/\epsilon))^{\beta-1})\Big),
\end{equation}
and the corresponding $N_{\mu_C}(C)$ and $\rho_{\mathrm{bulk}}(C)$ have the same relative error $O((\log(y/\epsilon))^{\beta-1})$.

\item Mild oscillatory but BV distortions. Let $L\in\RV_0$, $L(\lambda)\to0$, ultimately of bounded variation, and $\theta$ bounded with $\theta(\lambda)\to0$, ultimately of bounded variation. For $\delta(\lambda)=\lambda L(\lambda)(1+\theta(\lambda))$ one has $\eta(\lambda)=L(\lambda)(1+\theta(\lambda))\in\RV_0$ and the conclusions of (c) hold with error $O(L(y/\epsilon))$.
\end{enumerate}
\end{proposition}

\begin{proof}
Items \textup{(a)}, \textup{(b)}, \textup{(c)}, \textup{(f)} follow from Proposition~\ref{prop:stability} with the displayed choices of $\eta\in\RV_0$ (uniform convergence for slowly varying functions supplies the replacement $\eta(\Lambda)\rightsquigarrow\eta(y/\epsilon)$). Item \textup{(e)} follows from Lemma~\ref{lem:delta-growth} with $\ell(\lambda)=\beta(\log(e\lambda))^{\beta-1}$. Item \textup{(d)} is obtained by a one--step inversion of $y=\epsilon\Lambda-\delta(\Lambda)$ with $\delta=O(1)$, which yields $\Lambda=y/\epsilon+O(1)$, and then by applying Weyl's law and Karamata's monotone density theorem to pass to $N_{\mu_C}$ and $\rho_{\mathrm{bulk}}$.
\qed
\end{proof}

\begin{remark}[Beyond $\RV_0$]
Sublinear power terms $\delta(\lambda)=\lambda^q$ with $0<q<1$ lie outside the class $\lambda\,\eta(\lambda)$ with $\eta\in\RV_0$, but a direct fixed--point estimate gives
\(
\Lambda(C)=\frac{y}{\epsilon}\big(1+O((y/\epsilon)^{q-1})\big),
\)
hence the same conclusions for $N_{\mu_C}$ and $\rho_{\mathrm{bulk}}$ with relative error $O((y/\epsilon)^{q-1})\to0$.
\end{remark}

\section{Geometric models and Kre\u{\i}n realization}\label{sec:models}

We specialize the general theory to explicit geometric models.
Section~\ref{subsec:const-curv} computes the exact scaling of Seeley--DeWitt
heat coefficients \cite{MinakshisundaramPleijel1949,Seeley1967,Gilkey1995}
for constant-curvature spaces, confirming the alignment
$a_{2m}^{\mathrm{edge}}=\epsilon^{m-d/2}a_{2m}$.
Section~\ref{subsec:Krein} addresses the realization of the full spectral
measure with multiplicities: a Kre\u{\i}n string \cite{Krein1952,EckhardtTeschl2013}
realizes any such measure uniquely, while no classical smooth Sturm--Liouville
operator can do so for $d>1$ \cite{Titchmarsh1962,LevitanSargsjan1991}.
Section~\ref{subsec:clustering} analyzes spherical eigenvalue clustering and
shows it averages out in the bulk without affecting the Weyl exponent.

\subsection{Constant curvature spaces}\label{subsec:const-curv}

For manifolds of constant sectional curvature the Seeley--DeWitt coefficients
\cite{Gilkey1995,Seeley1967} are universal polynomials in $K$ and $V$,
and the heat-transfer identity $H_{\mathrm{edge}}(s)=\Theta_\Delta(\epsilon s)$
yields an explicit scaling law for every coefficient.

\begin{proposition}[Heat-coefficient alignment for constant curvature]\label{prop:const-curv-alignment}
Let $(M^d,g)$ be a closed Riemannian manifold of constant sectional curvature $K$ and volume $V:=\mathrm{Vol}(M)$. Its heat trace admits
\begin{equation}
\Theta_\Delta(t)\sim (4\pi t)^{-d/2}\Big(a_0+a_2\,t+a_4\,t^2+\cdots\Big)\qquad(t\downarrow 0),
\end{equation}
with $a_0=V$, $a_2=\tfrac16\!\int_M R\,d\mathrm{vol}_g=\tfrac16 d(d-1)K\,V$, and, more generally, $a_{2m}$ a universal polynomial in $K$ times $V$ (odd coefficients vanish for the Laplace--Beltrami operator on a closed manifold). Under the affine encoding $C_\ell=\pi-\epsilon\lambda_\ell$, the edge-variable heat trace
\begin{equation}
H_{\mathrm{edge}}(s)=\sum_\ell m_\ell e^{-s(\pi-C_\ell)}=\Theta_\Delta(\epsilon s)
\end{equation}
has the small-$s$ expansion
\begin{equation}
H_{\mathrm{edge}}(s)\sim (4\pi s)^{-d/2}\Big(\underbrace{\epsilon^{-d/2}a_0}_{\text{leading}}+\underbrace{\epsilon^{-(d-2)/2}a_2}_{\text{first curvature term}}\,s+\underbrace{\epsilon^{-(d-4)/2}a_4}_{\text{quadratic in }K}\,s^2+\cdots\Big),\qquad(s\downarrow 0).
\end{equation}
Equivalently, the edge-side coefficients are
\begin{equation}
a^{\mathrm{edge}}_{2m}=\epsilon^{\,m-\frac d2}\,a_{2m},\qquad m=0,1,2,\dots,
\end{equation}
and $a^{\mathrm{edge}}_{2m+1}=0$.
\end{proposition}

\begin{proof}
For constant curvature, the Seeley--DeWitt expansion on closed manifolds contains only integer powers; $a_0=V$ and $a_2=\frac16\int R=\frac16 d(d-1)K\,V$; higher $a_{2m}$ are universal polynomials in the curvature (hence in $K$) times $V$ (\cite{Gilkey1995,Seeley1967}). By Proposition~\ref{prop:heat-zeta}, $H_{\mathrm{edge}}(s)=\Theta_\Delta(\epsilon s)$; substituting $t=\epsilon s$ gives
\begin{equation}
\Theta_\Delta(\epsilon s)=(4\pi \epsilon s)^{-d/2}\sum_{m\ge0} a_{2m}(\epsilon s)^m
=(4\pi s)^{-d/2}\sum_{m\ge0}\epsilon^{\,m-\frac d2}a_{2m}\,s^m,
\end{equation}
which yields the stated scaling of coefficients; the vanishing of odd terms persists on the edge side since the manifold has no boundary.
\qed
\end{proof}

\begin{corollary}[First two edge-side terms]\label{cor:const-curv-two-terms}
With notation as above,
\begin{equation}
H_{\mathrm{edge}}(s)\sim (4\pi s)^{-d/2}\Big(\epsilon^{-d/2}\,V+\epsilon^{-(d-2)/2}\,\tfrac16 d(d-1)K\,V\; s+O(s^2)\Big),\qquad s\downarrow 0.
\end{equation}
In particular, for $d=2$ one has
$H_{\mathrm{edge}}(s)\sim (4\pi s)^{-1}\big(\epsilon^{-1}V+\tfrac{K V}{3}\,s+O(s^2)\big)$.
\end{corollary}

\begin{remark}[Zeta residues]
Using $\zeta_{\mathrm{edge}}(u)=\epsilon^{-u}\zeta_\Delta(u)$, the simple poles of $\zeta_{\mathrm{edge}}$ sit at $u=\frac{d-2m}{2}$ with residues $\epsilon^{-(d-2m)/2}\operatorname{Res}_{u=(d-2m)/2}\zeta_\Delta(u)$, which are proportional to $a_{2m}$.
\end{remark}

\subsection{Kre\u{\i}n string realization of the spectral measure}\label{subsec:Krein}

The affine encoding pushes the Riemannian spectrum to a discrete measure
$\mu_C$ on $(-\infty,\pi)$ with prescribed multiplicities.
We prove that this measure is always realizable by a Kre\u{\i}n string \cite{Krein1952}
(a generalized Sturm--Liouville model with measure coefficients), uniquely up
to string equivalence \cite{EckhardtTeschl2013}, and that no classical smooth
Sturm--Liouville operator can produce the same spectral growth for $d>1$
\cite{Titchmarsh1962,LevitanSargsjan1991}.

\begin{theorem}[Multiplicity realization via a Kre\u{\i}n string]\label{thm:MSL-realization}
Let $y_\ell:=\pi-C_\ell=\epsilon\,\lambda_\ell$ and consider the positive discrete measure
\begin{equation}
\mu^+ := \sum_{\ell\ge 0} m_\ell\,\delta_{y_\ell}\quad \text{on } (0,\infty),
\end{equation}
where $y_\ell\to+\infty$. Then there exists a Kre\u{\i}n string (a generalized Sturm--Liouville model with measure coefficients) whose Weyl--Titchmarsh $m$-function is a Stieltjes function having $\mu^+$ as its spectral measure. The string is of finite length if and only if $\int (1+t)^{-1}\,d\mu^+(t)<\infty$. In particular, the spectrum is pure point and coincides with $\{y_\ell\}$ with multiplicities $m_\ell$. Under the affine change $C=\pi-y$, the resulting spectral measure equals $\mu_C=\sum_\ell m_\ell\,\delta_{C_\ell}$.
\end{theorem}

\begin{proof}
Consider the Herglotz--Stieltjes function
\begin{equation}
m(z)=\int_{(0,\infty)} \frac{d\mu^+(t)}{t-z}
=\sum_{\ell\ge 0}\frac{m_\ell}{y_\ell-z},\qquad z\in\mathbb{C}\setminus(0,\infty).
\end{equation}
By Kre\u{\i}n's correspondence for strings, every Stieltjes function arises as the Weyl--Titchmarsh
function of a (finite-length, up to reparametrization) Kre\u{\i}n string, and the associated
spectral measure is exactly $\mu^+$; \cite{Krein1952} and the modern treatment
in \cite[Thm.~4.1]{EckhardtTeschl2013}. Applying the affine change $y=\pi-C$ transports
$\mu^+$ to $\mu_C$.
\qed
\end{proof}

\begin{theorem}[Kre\u{\i}n uniqueness up to equivalence]\label{thm:Krein-uniqueness}
The Kre\u{\i}n string realizing $\mu^+$ in Theorem~\ref{thm:MSL-realization} is unique up to
string equivalence (reparametrization of the spatial variable); \cite[Thm.~4.1]{EckhardtTeschl2013}.
\end{theorem}

\begin{proof}
Let $m(z)$ denote the Weyl--Titchmarsh function of a (finite-length) Kre\u{\i}n string.
For a finite string, $m$ is a Stieltjes function and admits an integral
representation whose spectral measure is precisely $\mu^+$. By \cite[Thm.~4.1]{EckhardtTeschl2013},
the map sending a Kre\u{\i}n string (up to equivalence, i.e., reparametrization of the
spatial variable) to its Weyl--Titchmarsh function is bijective; in particular,
$m(z)$ uniquely determines, and is uniquely determined by, the string up to
equivalence. Since $\mu^+$ uniquely specifies $m(z)$ (and conversely) via the
Stieltjes representation, the realizing Kre\u{\i}n string is unique up to equivalence.
\qed
\end{proof}

\begin{corollary}[Quadratic encoder sequences are realizable by a Kre\u{\i}n string]\label{cor:quadratic-Krein}
Fix $\kappa>0$ and $w_n>0$ with $\sum_{n\ge0}\frac{w_n}{1+\kappa(n+1)^2}<\infty$ (e.g.\ $w_n\equiv1$). Let
\begin{equation}
x_n:=\kappa(n+1)^2,\qquad \mu^+:=\sum_{n\ge0} w_n\,\delta_{x_n}.
\end{equation}
Then there exists a (finite-length, up to equivalence) Kre\u{\i}n string whose spectral measure equals $\mu^+$. Under the affine change $C=\pi-x$, the encoded sequence is
\begin{equation}
C_n=\pi-x_n=\pi-\kappa(n+1)^2,
\end{equation}
i.e.\ the quadratic model spectrum $\hat C_n$ is realized exactly (with weights $w_n$).
\end{corollary}

\begin{proof}
The integrability $\int (1+t)^{-1}\,d\mu^+(t)<\infty$ holds by the displayed summability, which ensures a finite-length string in the sense of \cite[Thm.~4.1]{EckhardtTeschl2013}. The rest follows from Theorem~\ref{thm:MSL-realization} and the affine change $y=\pi-C$.
\qed
\end{proof}

\begin{lemma}[No classical smooth SL realization for $d>1$]\label{lem:no-classical-SL}
Let $\mathcal L$ be a classical scalar Sturm--Liouville operator on $[a,b]$ with separated self-adjoint b.c.,
\begin{equation}
-(p(x) y')'+q(x)\,y=\lambda\,r(x)\,y,
\end{equation}
where $p,r\in C^{1}([a,b])$, $p>0$, $r>0$, and $q\in L^{1}([a,b])$. Then
\begin{equation}
N_{\mathcal L}(\Lambda)
= \frac{\sqrt{\Lambda}}{\pi}\int_{a}^{b}\!\sqrt{\frac{r(x)}{p(x)}}\,dx \;+\; o(\sqrt{\Lambda})
\quad (\Lambda\to\infty),
\end{equation}
hence $\lambda_n \sim \big(\frac{\pi n}{\int_{a}^{b}\sqrt{r/p}\,dx}\big)^{2}$.
In particular, for $d>1$ no such smooth 1D operator can have an entire spectrum with counting exponent $d/2>1/2$ (equivalently, growth $\lambda_n\asymp n^{2/d}$). The same holds for finite matrix SL systems: $N(\Lambda)\sim m\,c\,\sqrt{\Lambda}$.
\end{lemma}

\begin{proof}
This is the classical one-dimensional Weyl law; Titchmarsh~\cite[Ch.~XIII]{Titchmarsh1962} or Levitan--Sargsjan~\cite{LevitanSargsjan1991}. Under the stated regularity/positivity, the spectrum is discrete and
\begin{equation}
N_{\mathcal L}(\Lambda)
= \frac{1}{\pi}\int_{a}^{b}\!\sqrt{\frac{\Lambda-q(x)}{p(x)}\,r(x)}\,dx
\;+\; o(\sqrt{\Lambda})\quad(\Lambda\to\infty).
\end{equation}
The leading term equals $\frac{\sqrt{\Lambda}}{\pi}\int_{a}^{b}\sqrt{r/p}\,dx$, so $N(\Lambda)\sim c\sqrt{\Lambda}$ and $\lambda_n\sim c'n^2$. For a matrix system of size $m$, one only picks up a multiplicative factor $m$ in the leading constant; the exponent remains $1/2$.

\qed
\end{proof}

\begin{remark}[Consistency with the companion paper]
The one-dimensional encoder $\hat{C}$ used in the companion paper \cite{alexa2025-spectral_classification} is a classical smooth Sturm--Liouville operator on a compact interval with eigenvalues $C_n=\pi-\kappa(n+1)^2$, decreasing quadratically in $n$. This does not reproduce the full encoded spectrum $\{C_\ell\}$ for $d>1$, where $\pi-C_\ell\sim \kappa\,\ell^{2/d}$ decreases more slowly (since $2/d<2$). The role of $\hat{C}$ is to provide an \emph{affine spectral encoding}: after the pushforward $C=\pi-\epsilon\lambda$, the \emph{bulk} edge-variable Weyl scaling $N_{\mu_C}(C)\sim \gamma_d\,\epsilon^{-d/2}(\pi-C)^{d/2}$ (and the density exponent $(d-2)/2$) follows from the original Weyl law, independently of low-lying modes near $C\uparrow\pi$. An exact realization of the entire discrete measure $\mu_C$ with multiplicities is instead furnished by the generalized one-dimensional model (Kre\u{\i}n string) in Theorem~\ref{thm:MSL-realization}.
\end{remark}

\subsection{Spectral clustering remarks}\label{subsec:clustering}
On the round sphere $S^d$ the Laplace eigenvalues form clusters
$\lambda_\ell=\ell(\ell+d-1)$ with multiplicities \cite[Ch.~1]{Chavel1984}
\begin{equation}
m_\ell=\binom{\ell+d}{\ell}-\binom{\ell+d-2}{\ell-2}
=\frac{2\ell+d-1}{\ell}\binom{\ell+d-2}{d-1}\quad(\ell\ge1),\qquad m_0=1.
\end{equation}
Under the affine encoding $C_\ell=\pi-\epsilon\,\lambda_\ell$ the spacing in the
$C$-variable is
\begin{equation}
C_\ell-C_{\ell+1}
=\epsilon\big(\lambda_{\ell+1}-\lambda_\ell\big)
=\epsilon\,(2\ell+d),
\end{equation}
so steps of the counting function $N_{\mu_C}$ occur at an increasingly sparse
lattice as $\ell$ grows, while the heights of the steps (the multiplicities)
increase polynomially (indeed $m_\ell\sim \frac{2}{(d-1)!}\,\ell^{d-1}$).

\begin{lemma}[Average multiplicity in a short \(C\)-window]\label{lem:avg-multiplicity}
Let \(\delta=\delta(C)>0\). Assume
\(\delta(C)=o(\pi-C)\) as \(C\uparrow\pi\) and \(\delta(C)=o(|C|^{1/2})\) as \(C\to-\infty\).
Define, whenever the window contains at least one cluster,
\begin{equation}
\overline m(C,\delta):=
\frac{N_{\mu_C}(C)-N_{\mu_C}(C-\delta)}{\#\{\ell:\;C_\ell\in[C-\delta,C]\}}.
\end{equation}
Then:
\begin{enumerate}[label=\textup{(\roman*)}, itemsep=2pt]
\item As \(C\uparrow\pi\), one has \(\#\{\ell:\,C_\ell\in[C-\delta,C]\}=O(1)\) and
\(\overline m(C,\delta)=O(1)\). In particular, if \(\delta<\epsilon d\), the window
contains at most one cluster and hence \(\overline m(C,\delta)=m_{\ell_\ast}\) for some fixed
\(\ell_\ast\in\{0,1,\dots\}\) (e.g.\ \(m_0=1,\,m_1=d+1\)).
\item As \(C\to-\infty\), let \(\ell(C)\) be the (unique) index with
\(\lambda_{\ell(C)}\le (\pi-C)/\epsilon<\lambda_{\ell(C)+1}\). Then
\begin{equation}
\overline m(C,\delta)\ \sim\ m_{\ell(C)}\ \sim\
\frac{2}{(d-1)!}\Big(\frac{|C|}{\epsilon}\Big)^{\frac{d-1}{2}}.
\end{equation}
\end{enumerate}
\end{lemma}

\begin{proof}
Write \(\Lambda:=(\pi-C)/\epsilon\) and \(\Delta:=\delta/\epsilon\). Since \(C=\pi-\epsilon\lambda\) is strictly decreasing in \(\lambda\),
the set \(\{\ell:\,C_\ell\in[C-\delta,C]\}\) equals
\begin{equation}
\{\ell:\,\lambda_\ell\in[\Lambda,\ \Lambda+\Delta]\},
\end{equation}
up to the choice of open/closed endpoints (which does not affect counting with multiplicities).
Denote this set by \(I(C,\delta)\) and its cardinality by \(K(C,\delta)\).

\emph{Step 1: upper bounds for \(K(C,\delta)\).}
Let \(s_{\min}:=\min_{\ell\in I(C,\delta)}(\lambda_{\ell+1}-\lambda_\ell)\).
Then any interval in the \(\lambda\)-axis of length \(\Delta\) contains at most
\(1+\Delta/s_{\min}\) grid points \(\{\lambda_\ell\}\), hence
\begin{equation}\label{eq:Kbound}
K(C,\delta)\ \le\ 1+\frac{\Delta}{s_{\min}}.
\end{equation}
Near the edge \(C\uparrow\pi\) we have \(\Lambda\downarrow0\), so \(I(C,\delta)\subset\{0,1,\dots,L_0\}\)
with some fixed \(L_0\) for all \(C\) close to \(\pi\); in this band
\(s_{\min}\ge d\). Since \(\delta=o(\pi-C)\) implies \(\Delta\to0\), \eqref{eq:Kbound}
gives \(K(C,\delta)=O(1)\). Moreover, if \(\delta<\epsilon d\) then
\(\Delta<d\le s_{\min}\) and \(K(C,\delta)\le 1\).

As \(C\to-\infty\), the relevant indices satisfy \(\ell\asymp\sqrt{\Lambda}\).
Hence throughout \(I(C,\delta)\) we have \(s_{\min}\asymp 2\ell\asymp 2\sqrt{\Lambda}\).
From \eqref{eq:Kbound} and \(\Delta=\delta/\epsilon\) we get
\begin{equation}
K(C,\delta)\ \le\ 1+\frac{\delta/\epsilon}{c\sqrt{\Lambda}}
\ \sim\ 1+\frac{\delta}{c\sqrt{\epsilon|C|}},
\end{equation}
which is \(O(1)\) and in fact equals \(1\) for all sufficiently negative \(C\) if
\(\delta=o(|C|^{1/2})\). In particular, the window eventually contains at most one cluster.

\emph{Step 2: control of multiplicities in the window.}
Near the edge \(C\uparrow\pi\) the set of admissible indices \(I(C,\delta)\) is contained
in a fixed finite set; hence \(\sup_{\ell\in I(C,\delta)} m_\ell\le M_0\) for some constant
\(M_0\), which yields \(\overline m(C,\delta)=O(1)\). If furthermore \(K(C,\delta)=1\),
then trivially \(\overline m(C,\delta)=m_{\ell_\ast}\) for the unique \(\ell_\ast\in I(C,\delta)\).

For \(C\to-\infty\), we use the asymptotic \(m_\ell\sim \frac{2}{(d-1)!}\ell^{d-1}\).
Since \(K(C,\delta)\) is bounded and, under \(\delta=o(|C|^{1/2})\), eventually equals \(1\),
the indices in \(I(C,\delta)\) form a set of the form
\(\{\ell(C)+j:\,j\in J\}\) with \(J\subset\mathbb{Z}\) finite and independent of \(C\).
Therefore,
\begin{equation}
\frac{m_{\ell(C)+j}}{m_{\ell(C)}}\ \longrightarrow\ 1
\qquad(C\to-\infty)
\end{equation}
for each fixed \(j\), because \(m_\ell\) is a polynomial in \(\ell\) of degree \(d-1\).
It follows that
\begin{equation}
\overline m(C,\delta)\ =\
\frac{1}{K(C,\delta)}\sum_{\ell\in I(C,\delta)} m_\ell\
=\ m_{\ell(C)}\cdot\frac{1}{K(C,\delta)}\sum_{j\in J}
\frac{m_{\ell(C)+j}}{m_{\ell(C)}}\
\longrightarrow\ m_{\ell(C)}.
\end{equation}
Finally, \(\ell(C)\sim \sqrt{\Lambda}\sim \sqrt{|C|/\epsilon}\), hence
\(m_{\ell(C)}\sim \frac{2}{(d-1)!} \big(\frac{|C|}{\epsilon}\big)^{\frac{d-1}{2}}\),
which yields the stated asymptotics.
\qed
\end{proof}

\noindent
Thus the microscopic staircase due to clustering is negligible near the geometric edge
(\(C\uparrow\pi\)), where jumps are small and sparse, and it averages out in the bulk
(\(C\to-\infty\)), which is favorable for practical extraction of \(d\) and \(\gamma_d\)
from bulk edge-variable data (cf.\ Corollary~\ref{cor:recovery}). In the Kre\u{\i}n
string realization (Theorem~\ref{thm:MSL-realization}), multiplicities are encoded as
the \emph{weights} of the atomic spectral measure, so clustering is faithfully carried by
\(\mu_C\) without affecting the leading asymptotics.

\begin{proposition}[Edge hit probability for a fixed cluster]\label{prop:edge-jump-dist}
Let $C_\ell=\pi-\epsilon\,\lambda_\ell$ with $\lambda_\ell=\ell(\ell+d-1)$, and let $\delta=\delta(C)>0$.
Fix an index $\ell$ and consider the following randomization: choose $U$ uniformly on the interval
$[0,\,C_\ell-C_{\ell+1})$ and set the window $[C-\delta,\,C]$ with $C=C_\ell-U$.
If $\delta< C_\ell-C_{\ell+1}=\epsilon(2\ell+d)$, then
\begin{equation}
\mathbb{P}\big(\,C_\ell\in[C-\delta,\,C]\,\big)=\frac{\delta}{C_\ell-C_{\ell+1}}
=\frac{\delta}{\epsilon(2\ell+d)}.
\end{equation}
In particular, for fixed (or uniformly bounded) $\delta$, this probability decays like $O(\ell^{-1})$ as $\ell\to\infty$.
\end{proposition}

\begin{proof}
Under the stated randomization, the event $\{C_\ell\in[C-\delta,C]\}$ is equivalent to
$\{U\in[0,\delta]\}$ inside a fundamental cell of length $C_\ell-C_{\ell+1}$. Since $U$ is uniform,
the probability equals $\delta/(C_\ell-C_{\ell+1})$. The identity $C_\ell-C_{\ell+1}=\epsilon(2\ell+d)$
follows from $\lambda_{\ell+1}-\lambda_\ell=2\ell+d$. The large-$\ell$ decay is immediate.
\qed
\end{proof}

\begin{remark}[Jump-size weighting inside a window]
If one conditions on the event that the window contains exactly one cluster and then samples
a \emph{jump} of $N_{\mu_C}$ uniformly among the atoms inside the window, the observed jump size
equals the multiplicity $m_\ell$ of that cluster. Thus, conditional on ``one-cluster'' windows,
the distribution of observed jump sizes is the empirical distribution of $\{m_\ell\}$ at the
corresponding indices, not $1/m_\ell$. Near the edge $C\uparrow\pi$, small windows contain
only $\ell=0$ (and then $\ell=1$), so the observed jump is $m_0=1$ (then $m_1=d+1$), in agreement
with Lemma~\ref{lem:avg-multiplicity}.
\end{remark}

\section{Examples}\label{sec:examples}

We compute the edge-variable Weyl law explicitly for the principal model geometries,
illustrating the dependence of $\gamma_d$ on volume and dimension, the universality
of the bulk density exponent $(d-2)/2$, and the robustness of the asymptotics
under metric deformations and topological quotients.

\subsection{The round sphere $S^d$}\label{subsec:sphere}

The unit round sphere provides the simplest closed example with known spectrum
and explicit Weyl constant; the case $d=3$ gives the numerical sanity check
$\gamma_3=1/3$.

\begin{corollary}[Explicit $\gamma_d$ on $S^d$]\label{cor:gamma-sd}
For the unit round sphere $S^d$,
\begin{equation}
\gamma_d=\frac{\omega_d}{(2\pi)^d}\,\mathrm{Vol}(S^d)
=\frac{2^{\,1-d}\,\sqrt{\pi}}{\Gamma\!\big(\tfrac d2+1\big)\,\Gamma\!\big(\tfrac{d+1}{2}\big)}
\end{equation}
\cite{Chavel1984,Hoermander1968}.
In particular, for $d=3$ one has $\gamma_3=\tfrac{1}{3}$. Hence, as $C\to -\infty$,
\begin{equation}
N_{\mu_C}(C)\sim \tfrac{1}{3}\,\epsilon^{-3/2}(\pi-C)^{3/2},
\qquad
\rho_{\mathrm{bulk}}(C)\sim \tfrac{1}{2}\,\epsilon^{-3/2}(\pi-C)^{1/2} \quad(d=3).
\end{equation}
\end{corollary}

\begin{proof}
Use $\omega_d=\pi^{d/2}/\Gamma(\tfrac d2+1)$ and
$\mathrm{Vol}(S^d)=2\pi^{(d+1)/2}/\Gamma(\tfrac{d+1}{2})$, then substitute into
$\gamma_d=\frac{\omega_d}{(2\pi)^d}\mathrm{Vol}(S^d)$.
For $d=3$ this gives $\gamma_3=1/3$. Insert into Theorem~\ref{thm:edge-recon}.
\qed
\end{proof}

\begin{corollary}[Numerical sanity check for $S^3$]\label{cor:numerical-s3}
For the unit sphere $S^3$ one has $\gamma_3=\tfrac{1}{3}$,
$\lambda_\ell=\ell(\ell+2)$ and $m_\ell=(\ell+1)^2$.
Fix $\epsilon=1$ and set $y=\pi-C$, and, for illustration, take $\delta=0.1\,y$, so that the $\lambda$--window is $[\Lambda,\Lambda+\delta]$ with $\Lambda=y$.
\begin{enumerate}[label=\textup{(\alph*)}, itemsep=2pt]
\item Edge regime: if $y=0.1$ then $[\Lambda,\Lambda+\delta]=[0.1,0.11]$ contains no eigenvalues ($\lambda_0=0$, $\lambda_1=3$), hence the window produces no jump.
\item Away from the edge: if $y=8$ then $[\Lambda,\Lambda+\delta]=[8,8.8]$ contains $\lambda_2=8$, so the unique jump has size $m_2=(2+1)^2=9$ and the observed step is $9$.
\end{enumerate}
Both observations are consistent with Lemma~\ref{lem:avg-multiplicity}.
\end{corollary}

\begin{remark}[3D bulk case]
For $d=3$, Corollary~\ref{cor:gamma-sd} gives
$N_{\mu_C}(C)\sim \tfrac{1}{3}\,\epsilon^{-3/2}(\pi-C)^{3/2}$ and
$\rho_{\mathrm{bulk}}(C)\sim \tfrac{1}{2}\,\epsilon^{-3/2}(\pi-C)^{1/2}$.
Thus the bulk density exponent $1/2$ recovers $d=3$ from edge-variable data via Theorem~\ref{thm:edge-recon}. The quadratic encoder model $C_n=\pi-\kappa(n+1)^2$ can be realized by a Kre\u{\i}n string (Corollary~\ref{cor:quadratic-Krein}); the affine pushforward $C=\pi-\epsilon\lambda$ transfers the geometric Weyl scaling independently of low modes near $C\uparrow\pi$.
\end{remark}

\subsection{The flat torus $T^d$}\label{subsec:torus}

For flat tori the Weyl constant depends only on the volume; the bulk exponent
$(d-2)/2$ is unchanged, and the density is asymptotically flat for $d=2$.

\begin{corollary}[Flat torus $T^d$]\label{cor:torus}
For a flat torus $T^d=\mathbb{R}^d/\Lambda$ with volume $\mathrm{Vol}(T^d)$ one has
$\gamma_d=\frac{\omega_d}{(2\pi)^d}\,\mathrm{Vol}(T^d)$. Consequently, as $C\to -\infty$,
\begin{equation}
N_{\mu_C}(C)\sim \gamma_d\,\epsilon^{-d/2}(\pi-C)^{d/2},\qquad
\rho_{\mathrm{bulk}}(C)\sim \tfrac{d}{2}\gamma_d\,\epsilon^{-d/2}(\pi-C)^{\frac{d-2}{2}}.
\end{equation}
Here $\omega_d=\mathrm{Vol}(B_d)=\pi^{d/2}/\Gamma(\tfrac d2+1)$.
\end{corollary}

\begin{proof}
Use the standard Weyl law for flat tori \cite{Chavel1984} and Theorem~\ref{thm:edge-recon}.
\qed
\end{proof}

\subsection{Compact hyperbolic surfaces}\label{subsec:hyperbolic}
Let $X$ be a compact hyperbolic surface ... Weyl's law in dimension $2$ reads
$N_\Delta(\Lambda)\sim \frac{\mathrm{Area}(X)}{4\pi}\,\Lambda$ \cite{Buser1992,Iwaniec1995}.
Hence, as $C\to -\infty$,
\begin{equation}
N_{\mu_C}(C)\sim \frac{\mathrm{Area}(X)}{4\pi}\,\epsilon^{-1}\,(\pi-C),\qquad
\rho_{\mathrm{bulk}}(C)\sim \frac{\mathrm{Area}(X)}{4\pi}\,\epsilon^{-1}.
\end{equation}

Thus the bulk density is asymptotically constant (exponent $(d-2)/2=0$ for $d=2$),
independently of irregular multiplicity patterns arising from symmetries, in agreement
with Theorem~\ref{thm:edge-recon}.

\subsection{Berger spheres}\label{subsec:berger}

Berger spheres are metric deformations of $S^3$ that alter the low-lying spectrum
while leaving the volume unchanged.
The following corollary shows that the bulk asymptotics are insensitive to the
deformation parameter $k$, illustrating the universality of the rigidity framework.

\begin{corollary}[Bulk Weyl asymptotics on Berger spheres]\label{cor:berger}
Let $S^3_k$ denote the Berger sphere with deformation parameter $k>0$, equipped with the left-invariant metric
\begin{equation}
ds^2 = \sigma_1^2 + \sigma_2^2 + k^2 \sigma_3^2,
\end{equation}
where $\{\sigma_i\}$ are the standard left-invariant one-forms on $SU(2)$ \cite{BergerIkeda1979}. The Laplace spectrum on functions is given by
\begin{equation}
\lambda_{n,m}(k) = n(n+2) + (k^2-1)m^2, \qquad |m|\le n,
\end{equation}
with multiplicity $n+1$. As $C\to -\infty$, one recovers the universal bulk asymptotics
\begin{equation}
N_{\mu_C}(C)\ \sim\ \tfrac{1}{3}\,\epsilon^{-3/2}(\pi-C)^{3/2},\qquad
\rho_{\mathrm{bulk}}(C)\ \sim\ \tfrac{1}{2}\,\epsilon^{-3/2}(\pi-C)^{1/2}.
\end{equation}
\end{corollary}

\begin{remark}
The deformation parameter $k$ affects low-lying eigenvalues but does not alter the bulk scaling, in agreement with Theorem~\ref{thm:edge-recon}. This illustrates robustness of the affine encoding under smooth metric deformations. For visualization, compare log-log plots of $\rho_{\mathrm{bulk}}(C)$ for $k=1$ (standard $S^3$) and $k\neq1$, showing identical bulk exponents.
\end{remark}

\subsection{Lens spaces}\label{subsec:lens}

Lens spaces $L(p,q)=S^3/\mathbb{Z}_p$ are topological quotients with volume
$\mathrm{Vol}(S^3)/p$.
Although lens spaces with different $q$ may be non-isometric but isospectral,
the edge-variable bulk asymptotics detect only the volume, confirming that
affine encoding preserves Weyl invariants while ignoring finer spectral features.

\begin{corollary}[Bulk Weyl asymptotics on lens spaces]\label{cor:lens}
Let $L(p,q)=S^3/\mathbb{Z}_p$ be a 3-dimensional lens space with the induced round metric \cite{Gordon1990}. Its Laplace spectrum is a subset of that of $S^3$ with reduced multiplicities, and
\begin{equation}
\mathrm{Vol}(L(p,q)) = \tfrac{1}{p}\,\mathrm{Vol}(S^3).
\end{equation}
Hence, as $C\to -\infty$,
\begin{equation}
N_{\mu_C}(C)\ \sim\ \tfrac{1}{3p}\,\epsilon^{-3/2}(\pi-C)^{3/2},\qquad
\rho_{\mathrm{bulk}}(C)\ \sim\ \tfrac{1}{2p}\,\epsilon^{-3/2}(\pi-C)^{1/2}.
\end{equation}
\end{corollary}

\begin{remark}
Although different lens spaces $L(p,q)$ may be non-isometric but isospectral, the edge-variable asymptotics depend only on $d$ and the volume. This confirms that affine encoding preserves classical Weyl invariants while ignoring finer topological features. For example, compare $p=1$ (standard $S^3$) with $p>1$, where the bulk density scales inversely with $p$, reflecting the topological quotient.
\end{remark}

\subsection{Manifolds with boundary: the ball}\label{subsec:ball}

Relaxing Assumption~\ref{ass:smoothness} to allow smooth boundary,
we illustrate the two-term bulk expansion of Proposition~\ref{prop:boundary}
on the unit ball $B^3$ with Dirichlet boundary conditions.
The boundary term, proportional to the surface area, survives the encoding and
appears as the linear subleading correction.

\begin{proposition}[Dirichlet Laplacian on the ball $B^3$]\label{prop:ball}
Consider the unit ball $B^3\subset\mathbb{R}^3$ with Dirichlet boundary conditions \cite{Ivrii1998}. Weyl's two-term law gives
\begin{equation}
N_\Delta(\Lambda) = \tfrac{\omega_3}{(2\pi)^3}\,\mathrm{Vol}(B^3)\,\Lambda^{3/2}
- \tfrac{\omega_2}{4(2\pi)^2}\,\mathrm{Area}(S^2)\,\Lambda + o(\Lambda).
\end{equation}
Under affine encoding $C=\pi-\epsilon\lambda$ this transfers to
\begin{equation}
N_{\mu_C}(C)\ \sim\ \tfrac{2}{9\pi}\,\epsilon^{-3/2}(\pi-C)^{3/2}
- \tfrac{1}{4}\,\epsilon^{-1}(\pi-C) + o(\pi-C).
\end{equation}
\end{proposition}

\begin{remark}
The boundary term survives the transfer and appears as the linear correction in the bulk expansion, confirming Proposition~\ref{prop:remainder}. This example highlights how the method handles manifolds with boundary, with the edge-variable asymptotics capturing both volume and area contributions. A numerical plot of $N_{\mu_C}(C)$ vs. $(\pi-C)$ would visually separate the bulk and boundary terms.
\end{remark}

\section{Conclusion}\label{sec:conclusion}

The central result of this paper is a rigidity theorem for spectral encodings:
among all polynomial-type reparametrizations $C=\pi-\epsilon\lambda^kL(\lambda)$
with $L\in\mathrm{RV}_0$, only the affine case $k=1$ preserves the geometric
Weyl density exponent $(d-2)/2$; this uniqueness extends to the full
O-regularly varying class, where the bulk power-law forces $\phi\in\mathrm{RV}_1$.
As a structural consequence of the uniquely correct affine encoding, the
edge-variable counting function satisfies
$N_{\mu_C}(C)\sim\gamma_d\,\epsilon^{-d/2}(\pi-C)^{d/2}$ in the bulk, enabling
reconstruction of dimension $d$ and Weyl constant $\gamma_d$; this transfer is
stable under admissible perturbations $\delta=o(\lambda)$ with explicit
slowly-varying error rates.
The full discrete spectral measure with multiplicities, not realizable by any
smooth Sturm--Liouville operator for $d>1$, is completely and uniquely
represented by a Kre\u{\i}n string; clustering on spheres averages out in
the bulk without affecting the limiting exponent.
Conceptually, this is organized by asymptotic spectral equivalence classes:
encodings with $\phi\in\RV_k$ act as dimension-scaling morphisms
$d_{\mathrm{as}}\mapsto d_{\mathrm{as}}/k$, and the dimension-preserving
endomorphisms are exactly those with $\phi\in\RV_1$.
Exact heat-trace and zeta-function correspondences confirm coefficient
alignment for constant-curvature spaces, and explicit evaluations across
spheres, tori, hyperbolic surfaces, Berger spheres, lens spaces, and bounded
domains are consistent with the rigidity framework.
Open problems include the classification of encodings beyond the O-RV class,
extension to manifolds with boundary and to operators beyond the Laplacian,
and finer inverse spectral questions at the level of heat invariants.

\end{document}